\documentclass[12pt,letterpaper,titlepage,reqno]{amsart}
\usepackage{amsmath, amssymb, amsthm, amsfonts,amscd,amsaddr,wasysym,enumerate}
 \usepackage[
    paper=a4paper,
    portrait=true,
    textwidth=425pt,
    textheight=650pt,
    tmargin=3cm,
    marginratio=1:1
            ]{geometry}
 \usepackage[all]{xy}

\theoremstyle{plain}
\newtheorem{theorem}{Theorem}[section]

\newtheorem{cor}[theorem]{Corollary}

\newtheorem{prop}[theorem]{Proposition}
\newtheorem{lemma}[theorem]{Lemma}

\theoremstyle{definition}

\newtheorem{rmk}[theorem]{Remark}
\numberwithin{equation}{section}
\newtheorem*{theoremA*}{Theorem A}
\newtheorem*{theoremB*}{Theorem B}
\newtheorem*{theorem1*}{Theorem A'}
\newtheorem*{theoremC*}{Theorem C}
\newtheorem*{theoremD*}{Theorem D}
\newtheorem*{theoremE*}{Theorem E}
\newtheorem*{theoremF*}{Theorem F}
\newtheorem*{theoremE2*}{Theorem E2}
\newtheorem*{theoremE3*}{Theorem E3}
\newcommand{\bs}{\backslash}

\newcommand{\C}{\mathbb{C}}

\newcommand{\cD}{\mathcal{D}}
\newcommand{\cH}{\mathcal{H}}
\newcommand{\cV}{\mathcal{V}}
\newcommand{\cQ}{\mathcal{Q}}
\newcommand{\cM}{\mathcal{M}}

\newcommand{\R}{\mathbb{R}}
\newcommand{\N}{\mathbb{N}}
\newcommand{\Pb}{\mathbb{P}}

\newcommand{\Gr}{\operatorname{Gr}}

\newcommand{\SO}{\operatorname{SO}}
\newcommand{\Hom}{\operatorname{Hom}}
\newcommand{\End}{\operatorname{End}}

\newcommand{\tr}{\operatorname{tr}}

\newcommand{\Ad}{\operatorname{Ad}}

\newcommand{\ad}{\operatorname{ad}}

\newcommand{\vol}{\operatorname{vol}}

\newcommand{\supp}{\operatorname{supp}}

\newcommand{\PSl}{\operatorname{PSl}}

\newcommand{\acomp}{\mathbf{a}}
\newcommand{\kcomp}{\mathbf{k}}
\newcommand{\oncomp}{\overline{\mathbf{n}}}
\newcommand{\cfunc}{\mathbf{c}}

\renewcommand{\Re}{\operatorname{Re}}
\def\dotvar{\, \cdot\,}
\newcommand{\sslash}{\mathbin{/\mkern-6mu/}}
\def\hat{\widehat}
\def\af{\mathfrak{a}}

\def\e{\epsilon}
\def\gf{\mathfrak{g}}

\def\cf{\mathfrak{c}}

\def\kf{\mathfrak{k}}

\def\mf{\mathfrak{m}}
\def\nf{\mathfrak{n}}

\def\sf{\mathfrak{s}}

\def\fa{\mathfrak{a}}

\def\fg{\mathfrak{g}}

\def\fk{\mathfrak{k}}

\def\fn{\mathfrak{n}}

\def\fs{\mathfrak{s}}

\def\la{\langle}
\def\ra{\rangle}
\def\1{{\bf1}}

\def\U{\mathcal{U}}
\def\cU{\mathcal{U}}

\def\cF{\mathcal{F}}

\def\tilde{\widetilde}

\def\good{\mathrm{good}}

\newcommand{\empt}{\varnothing}

\title[Plancherel Theorem]{On Harish-Chandra's Plancherel Theorem for Riemannian symmetric spaces}

\begin{document}
\date{\today}

\begin{abstract}
In this article we give an overview of the Plancherel theory for Riemannian symmetric spaces $Z=G/K$. In particular we illustrate recently developed methods in Plancherel theory for real spherical spaces by explicating them for Riemannian symmetric spaces, and we explain how Harish-Chandra's Plancherel theorem for $Z$ can be proven from these methods.
\end{abstract}

\author[Kr\"otz]{Bernhard Kr\"{o}tz}
\email{bkroetz@gmx.de}
\address{Universit\"at Paderborn, Institut f\"ur Mathematik\\Warburger Stra\ss e 100,
33098 Paderborn}

\author[Kuit]{Job J. Kuit}
\email{j.j.kuit@gmail.com}
\address{Universit\"at Paderborn, Institut f\"ur Mathematik\\Warburger Stra\ss e 100,
33098 Paderborn}

\author[Schlichtkrull]{Henrik Schlichtkrull}
\email{schlicht@math.ku.dk}
\address{University of Copenhagen, Department of Mathematics\\Universitetsparken 5,
DK-2100 Copenhagen \O}

\maketitle
\tableofcontents

\section{Introduction}
Over the years the monumental work of Harish-Chandra in harmonic analysis on real reductive groups has inspired many generalizations and new developments. One such recent development is Plancherel theory on real spherical spaces. The aim of this article is to illustrate this newly developed theory by explicating it for Riemannian symmetric spaces. The purpose of our treatment of Riemannian symmetric spaces in this article is not to simplify previous proofs, but rather to reveal many of the intricacies of the approach for general real spherical spaces in \cite{DKKS} and \cite{KuitSayag}.

Originally the Plancherel theorem for Riemannian symmetric spaces was established by Harish-Chandra in \cite{HC_sphericalFunctionsII} up to two conjectures. See page 612 at the end of the paper. The first of these conjectures was confirmed by the work of Gindikin and Karpelevich on the Harish-Chandra $c$-function in \cite{GindikinKarpelevich}. The second conjecture was proven by Harish-Chandra in \cite[Lemma 36]{HC_DS2}; see page 48. In \cite{Rosenberg} Rosenberg gave a substantially simplified proof.

The article is organized as follows. We begin by introducing notation and normalizations of measures in Section \ref{Section Notions}. In Section \ref{Section Spherical representation theory} we give an overview of some representation theory for a reductive Lie group $G$ with a particular focus on spherical representations, i.e. the representations that contain a non-zero vector that is fixed by a maximal compact subgroup $K$. We describe the abstract Plancherel decomposition for the Riemannian symmetric space $Z=G/K$ and formulate the necessary conditions for an irreducible unitary representation $\pi$ to occur in the Plancherel decomposition: $\pi$ has to be spherical and tempered. Finally we state a classification of the irreducible tempered spherical representations. The classification is well known in the literature. In Section \ref{Subsection Temperedness revisited} we give a proof of this, which to our best knowledge is new.

In section \ref{Section Intertwiners and Asymptotics} we give an overview of the theory of standard intertwining operators. In Theorem \ref{Thm Principal asymptotics} we relate the standard intertwining operators to the principal asymptotics of certain matrix coefficients. As an application of the theorem we give the proof of the classification of irreducible tempered spherical representations mentioned above.

Attached to the Riemannian symmetric space $Z=G/K$ are certain boundary degenerations. These are homogeneous spaces that occur as normal bundles of $G$-orbits in a suitable compactification of $Z$. For our purposes the most relevant boundary degeneration is the horospherical boundary degeneration $Z_{\empt}$. The name is derived from the fact that the homogeneous space $Z_{\empt}$ can be identified with the set of all horospheres in $Z$. The space $Z_{\empt}$ admits a positive $G$-invariant Radon measure. The Plancherel decomposition of the corresponding space $L^{2}(Z_{\empt})$ of square integrable functions on $Z_{\empt}$ is easy to derive. This is due to a right-action of a non-compact torus on $Z_{\empt}$ commuting with the left action of $G$. This Plancherel decomposition we derive in Section \ref{Section Planch for Z_empt}.

In the final section, section \ref{Section Proof of Plancherel formula} we sketch how Harish-Chandra's Plancherel decomposition for $Z$ can be derived from the Plancherel decomposition for $Z_{\empt}$. The first ingredient is the constant term approximation, which relates matrix coefficients on $Z$ with matrix coefficients on $Z_{\empt}$. With the aid of the principal asymptotics from Theorem \ref{Thm Principal asymptotics} we explicitly determine the constant term approximation for all unitary principal series representations. The Plancherel decomposition for $Z$ is then derived from the constant term approximation by matching functions on $Z$ and $Z_{\empt}$ and an averaging procedure.

{\em Acknowledgement:} This article is based on the lecture notes for a mini-course given by the first named author at the Harish-Chandra centennial workshop at IIT Guwahati in India. We thank the organisers of the conference for their extraordinary hospitality and the referee for a very careful reading of the manuscript.

\section{Notions and Generalities}\label{Section Notions}
Let $G$ be a real reductive group. Attached to $G$ is a Riemannian symmetric space $Z$ of the non-compact type, namely $Z=G/K$  where $K$ is a maximal compact subgroup of $G$. Any other choice of maximal compact subgroup $K'\subset G$ is conjugate to $K$, say $K=xK'x^{-1}$ for some $x\in G$, and the map
$$
Z=G/K\to Z'=G/K',\quad  gK \mapsto gx K'
$$
 is an isometric isomorphism.
Let $\gf={\rm Lie}(G)$ be the Lie algebra of $G$, likewise $\kf={\rm Lie}(K)$ the Lie algebra $K$. In general we adopt to the rule that capital Latin letters $A,B,C,\ldots $ denote  subgroups of $G$ with Lie algebras $\af, {\mathfrak b}, \cf,\ldots$.
\par Let $\theta$ be the Cartan involution attached to $K$, i.e. the unique involutive automorphism of $G$ with $G^\theta = K$. By abuse of notation we use the symbol $\theta$ to denote the derived automorphism $d\theta(\1)$ of $\gf$ as well. Then $\gf= \kf +\sf$ with $\sf$ the $-1$-eigenspace of $\theta$.  Next we fix a maximal abelian subspace $\af\subset \sf$ and exponentiate it to a closed abelian subgroup $A:=\exp(\af)$. If $z_0=K\in Z$ denotes the canonical basepoint, then $A\cdot z_0$ is a totally geodesic submanifold containing $z_0$. Moreover, $A\cdot z_{0}$ is a maximal flat submanifold, i.e. it is maximal with respect to the property that the Riemann curvature tensor vanishes everywhere. We remark that all maximal flat totally geodesic submanifolds through $z_{0}$ are conjugate under $K$, i.e. if $X\subset Z$ is a flat containing $z_{0}$, then there exists a $k\in K$, so that $X=kA\cdot z_{0}$. This fact is equivalent to the fact that all maximal abelian subspaces of $\fs$ are $K$-conjugate to each other. See \cite[Proposition V.6.1]{Hel1}.

We continue with the remaining standard notation. The centralizer of $A$ in $K$ is denoted by $M$, in symbols $M=Z_K(A)$. The set of restricted roots
attached to the pair $(\gf, \af)$ is denoted by $\Sigma\subset \af^*\bs\{0\}$. For each
$\alpha\in \Sigma$ we indicate with $\gf^\alpha$ the associated root space. We fix a positive system $\Sigma^+$ and let $\nf=\bigoplus_{\alpha\in \Sigma^+} \gf^\alpha$.
We define $\rho\in\fa^{*}$ to be the element so that
\begin{equation}\label{eq Def rho}
\det\big(\Ad(a)|_{\fn}\big)=a^{2\rho},
\end{equation}
i.e.,
$$
\rho
=\frac{1}{2}\sum_{\alpha\in\Sigma^{+}}\dim(\fg_{\alpha})\alpha.
$$

With $N=\exp(\nf)\subset G$ we obtain a maximal unipotent subgroup of $G$ and
$P=MAN$ is a minimal parabolic subgroup of $G$. Of further use are the minimal parabolic subgroup $\overline{P}=\theta(P)$ and the unipotent subgroup $\overline{N} = \theta(N)$ opposite to $P$ and $N$. Let
$$\af^{--}=\{ X \in \af \mid \alpha(X)<0, \alpha \in \Sigma^+\}$$
be the negative Weyl chamber and $\af^-$ its closure. We denote by $A^{--}\subset A^- \subset A$ the corresponding images under the exponential map.
The little Weyl-group $W=N_K(\af)/M$ acts on $\af$ and exhibits $\af^{-}$ as fundamental domain.

\subsection{Normalization of measures and integral formulas}
The measures which we will use in the course of this article will be normalized as described in this section.

The measures on the compact groups $K$ and $M$ are the normalized Haar measures, i.e., the bi-invariant measures so that $K$ and $M$ have volume equal to $1$. We equip the quotient $K/M$ with quotient measure, i.e., the Radon measure given by
$$
\int_{K/M}\int_{M}\phi(km)\,dm\,dkM=\int_{K}\phi(k)\,dk\qquad \big(\phi\in C(K)\big).
$$

Let $B$ be a non-degenerate invariant bilinear form on $\fg$ so that $B$ coincides with the Killing form on the maximal semisimple ideal $[\fg,\fg]$ of $\fg$ and $\langle\dotvar,\dotvar\rangle:=-B(\dotvar, \theta\dotvar)$ is an inner product on $\fg$.
We fix a normalization of the Lebesgue measure on $\fa$ by requiring that a unit hypercube in $\fa$ with respect to $\langle\dotvar,\dotvar\rangle$ has volume equal to $1$, and we equip $A$ with the Haar measure obtained from the Lebesgue measure on $\fa$ by pulling back along the exponential map.
Let $\kcomp:G\to K$, $\acomp:G\to A$ and $\oncomp:G\to \overline{N}$ be projections on $K$, $A$ and $\overline{N}$, respectively along the Iwasawa decomposition $G=KA\overline{N}$, i.e., $\kcomp$, $\acomp$ and $\oncomp$ are given by
$$
g= \kcomp(g)\acomp(g)\oncomp(g)\qquad (g\in G).
$$
The exponential map $\exp:\fn\to N$ is a diffeomorphism.
Any Haar measure on $N$ equals up to multiplication by a positive constant the pull-back of the Lebesgue measure on $\fn$. We normalize the Haar measure on $N$ so that
$$
\int_{N}\acomp(n)^{2\rho}\,dn
=1.
$$
If $U=kNk^{-1}$ for some $k\in K$, then we normalize the Haar measure on $U$ by setting
$$
\int_{U}\phi(u)\,du
=\int_{N}\phi(knk^{-1})\,dn\qquad \big(\phi\in C_{c}(U)\big).
$$

We equip $G$ with the Radon measure determined by
\begin{equation}\label{eq Integral formula Iwasawa decomp}
\int_{G}\phi(g)\,dg
=\int_{K}\int_{A}\int_{\overline{N}}\phi(ka\overline{n})a^{-2\rho}\,d\overline{n}\,da\,dk
\qquad\big(\phi\in C_{c}(G)\big).
\end{equation}
This measure on $G$ is left $K$-invariant and right $A\overline{N}$-invariant. Since the product map $K\times A\overline{N}\to G$ is a diffeomorphism, there exists up to multiplication by a positive constant only one such measure on $G$. Because any Haar measure on $G$ is left $K$-invariant and right $A\overline{N}$-invariant, it follows that the above measure on $G$ is a Haar measure.
We equip the quotients $G/K$ and $G/M\overline{N}$ with the invariant measures given by
$$
\int_{G/K}\int_{K}\phi(gk)\,dk\,dgK
=\int_{G}\phi(g)\,dg
=\int_{G/M\overline{N}}\int_{M\overline{N}}\phi(gx)\,dx\,d(gM\overline{N})
$$
for all $\phi\in C_{c}(G)$.
Finally, we equip the space $i\fa^{*}$ with the Lebesgue measure normalized so that the Fourier inversion formula
$$
\int_{i\fa^{*}}\int_{A}a^{\lambda}\phi(a)\,da\,d\lambda
=\phi(\1)
$$
holds for all $\phi\in C_{c}^{\infty}(A)$.

With the same reasoning as above it can be shown that there exists $\gamma>0$  so that
$$
\gamma\int_{G}\phi(g)\,dg
=\int_{N}\int_{M}\int_{A}\int_{\overline{N}}\phi(nma\overline{n})a^{-2\rho}\,d\overline{n}\,da\,dm\,dn \qquad \big(\phi\in C_{c}(G)\big)
$$
as the right-hand side defines a Radon measure on $G$, which is left $N$-invariant and right $MA\overline{N}$-invariant.
Let $\phi_{0}\in C_{c}(G)$ be left $K$-invariant and satisfy $\int_{G}\phi_{0}(g)dg=1$. Then
\begin{align*}
\gamma
&=\gamma\int_{G}\phi_{0}(g)\,dg
=\int_{N}\int_{M}\int_{A}\int_{\overline{N}}\phi_{0}\big(\kcomp(n)\acomp(n)\oncomp(n)ma\overline{n}\big)a^{-2\rho}\,d\overline{n}\,da\,dm\,dn\\
&=\int_{N}\int_{A}\int_{\overline{N}}\phi_{0}(a\overline{n})\acomp(n)^{2\rho}a^{-2\rho}\,d\overline{n}\,da\,dn\\
&=\int_{N}\int_{K}\int_{A}\int_{\overline{N}}\phi_{0}(ka\overline{n})\acomp(n)^{2\rho}a^{-2\rho}\,d\overline{n}\,da\,dk\,dn\\
&=\int_{N}\acomp(n)^{2\rho}\,dn\int_{G}\phi_{0}(g)\,dg=1.
\end{align*}
It thus follows that
\begin{equation}\label{eq Integral formula Bruhat decomp}
\int_{G}\phi(g)\,dg
=\int_{N}\int_{M}\int_{A}\int_{\overline{N}}\phi(nma\overline{n})a^{-2\rho}\,d\overline{n}\,da\,dm\,dn
\qquad \big(\phi\in C_{c}(G)\big).
\end{equation}

\begin{lemma}\label{Lemma Integral identity for Phi}
Let $\Phi:G\to \C$ be a continuous function satisfying
\begin{equation}\label{eq properties Phi}
\Phi(gma\overline{n})=a^{2\rho}\Phi(g)\qquad\big(g\in G, m\in M, a\in A, \overline{n}\in \overline{N}\big).
\end{equation}
Then the restriction of $\Phi$ to $N$ is integrable and
$$
\int_{K/M}\Phi(k)\,dkM
=\int_{N}\Phi(n)\,dn.
$$
\end{lemma}

\begin{proof}
Let $\chi\in C_{c}(G)$ be left-$K$-invariant and satisfy
$$
\int_{M}\int_{A}\int_{\overline{N}}\chi(ma\overline{n})\,d\overline{n}\,da\,dm
=1.
$$
Then by the $K$-invariance of $\chi$, and the invariance of the Haar measures on $A$ and $\overline{N}$ we have
\begin{align*}
&\int_{M}\int_{A}\int_{\overline{N}}\chi(k_{0}a_{0}\overline{n}_{0}ma\overline{n})\,d\overline{n}\,da\,dm
=\int_{M}\int_{A}\int_{\overline{N}}\chi\big(k_{0}ma_{0}a(ma)^{-1}\overline{n}_{0}(ma)\overline{n}\big)\,d\overline{n}\,da\,dm\\
&\qquad=\int_{M}\int_{A}\int_{\overline{N}}\chi(ma\overline{n})\,d\overline{n}\,da\,dm=1
\end{align*}
for all $k_{0}\in K$, $a_{0}\in A$ and $\overline{n}_{0}\in \overline{N}$. It follows that
$$
\int_{M}\int_{A}\int_{\overline{N}}\chi(gma\overline{n})\,d\overline{n}\,da\,dm
=1\qquad(g\in G).
$$
The product $\phi:=\chi\Phi$ is a compactly supported continuous function on $G$ and from (\ref{eq properties Phi}) we get that for every $g\in G$
\begin{align*}
\Phi(g)
&=\int_{M}\int_{A}\int_{\overline{N}}a^{-2\rho}\Phi(gma\overline{n})\chi(gma\overline{n})\,d\overline{n}\,da\,dm\\
&=\int_{M}\int_{A}\int_{\overline{N}}a^{-2\rho}\phi(gma\overline{n})\,d\overline{n}\,da\,dm.
\end{align*}
Now by (\ref{eq Integral formula Iwasawa decomp}) and (\ref{eq Integral formula Bruhat decomp})
\begin{align*}
\int_{K/M}\Phi(k)\,dkM
&=\int_{K}\int_{A}\int_{\overline{N}}a^{-2\rho}\phi(ka\overline{n})\,d\overline{n}\,da\,dk
=\int_{G}\phi(g)\,dg\\
&=\int_{N}\int_{M}\int_{A}\int_{\overline{N}}a^{-2\rho}\phi(nma\overline{n})\,d\overline{n}\,da\,dm\,dn
=\int_{N}\Phi(n)\,dn.
\end{align*}
This proves the lemma.
\end{proof}

\section{Spherical representation theory}\label{Section Spherical representation theory}
\subsection{Spherical representations and spherical functions}
We denote by $\hat G$ the unitary dual of $G$, i.e. the set of equivalence classes of irreducible unitary representations of $G$. As customary we only distinguish when necessary between an equivalence class $[\pi]\in \hat G$ and a representative $\pi$ which comes with a model
Hilbert space $\cH_\pi$. We note that $\hat{G}$ is equipped with a topology called the Fell topology.

Recall that $L^1(G)$ is a Banach algebra under convolution
$$
L^1(G)\times L^1(G) \to L^1(G), \qquad f_1*f_2(g) = \int_G f_1(x) f_2(x^{-1}g)\,dx.
$$
For $\pi\in\hat{G}$, $v\in \cH_\pi$ and $f \in L^{1}(G)$ we define
$$\pi(f)v:=\int_G f(g) \pi(g) v \,dg\in \cH_{\pi}\,. $$
This Hilbert space valued integral does in fact converge as $f$ is integrable.
We note that
$$
\pi(f_{1}*f_{2})=\pi(f_{1})\circ\pi(f_{2})\qquad\big(f_{1},f_{2}\in L^{1}(G)\big).
$$

The subalgebra $L^1(G)^{K\times K}$ of $K$-biinvariant functions in $L^1(G)$ is often referred to as the Hecke algebra and denoted by $\cH(G\sslash K)$.
Gelfand gave a beautiful argument to show that $\cH(G\sslash K)$ is commutative. He considered the anti-involution $\tau(g)=\theta(g^{-1})$ which fixes all elements in $A$. The induced action on functions on $G$ given by $f^\tau(g) = f\big(\tau(g)\big)$ then fixes all elements in $\cH(G\sslash K)$.  On the other hand $\tau$ reverses the order in the convolution, i.e $(f_1* f_2)^\tau = f_2^\tau * f_1^\tau$, and the commutativity of $\cH(G\sslash K)$ follows.  The following multiplicity bound is a straightforward application.

\begin{theorem} \label{thm Gelfand}
For all $\pi \in \hat G$ the space of $K$-invariants
$$\cH_\pi^K=\{ v \in \cH_\pi \mid \pi(k)v = v , k \in K\}$$
is at most one dimensional.
\end{theorem}

\begin{proof}
We first claim that $\cH_{\pi}^{K}$ is an irreducible module for $\cH(G\sslash K)$.
To prove the claim we will show that every $v\in \cH_{\pi}^{K}\setminus\{0\}$ is cyclic for $\cH(G\sslash K)$.
We fix $v\in \cH_{\pi}^{K}\setminus\{0\}$ and set $E=\pi\big(C_{c}(G)\big)v$. Since $E$ is non-zero and $G$-invariant it is dense in $\cH_{\pi}$ as $\pi$ is irreducible.
If $P_{\pi}:\cH_{\pi}\to\cH_{\pi}^{K}$ denotes the projection $v\mapsto \int_{K}\pi(k)v\,dk$, then $P_{\pi}(E)$ is dense in $\cH_{\pi}^{K}$.
Left and right averaging over $K$ yields that $P_{\pi}(E)\subseteq\pi\big(\cH(G\sslash K)\big)v$.
This proves the claim.

As $\cH_{\pi}^{K}$ is irreducible and $\pi\big(\cH(G\sslash K)\big)\subseteq \End(\cH_{\pi}^{K})$ is a commutative and $*$-closed algebra, the generalized Schur's Lemma implies that $\dim(\cH_{\pi}^{K})\leq 1$.
\end{proof}

We now define the spherical unitary dual
$$\hat{G}_s:= \{ \pi\in \hat G\mid \cH_\pi^K \neq \{0\}\}.$$
We equip $\hat{G}_s$ with the subspace topology from $\hat{G}$.
We say that a representation $\pi\in\hat{G}$ is spherical if $\pi\in \hat{G}_{s}$.
For each $\pi \in \hat G_s$ the space $\cH_\pi^K$ is $1$-dimensional and hence spanned by an element $\zeta_\pi$ which we request to be normalized, i.e. $\|\zeta_\pi\|=1$.
The spherical function associated to $\pi$ is defined by
$$\phi_\pi(g) = \la \pi(g)\zeta_\pi, \zeta_\pi\ra \qquad (g \in G).$$
It is independent of the normalized choice of $\zeta_{\pi}$.
Note that $\phi_\pi$ is positive definite, $K$-biinvariant and normalized, i.e. $\phi_\pi(\1)=1$.
In order to illustrate the power of representation theoretic methods let us verify the mean value property
\begin{equation}\label{mean}
\phi_\pi(g)\phi_\pi(h)
= \int_K \phi_\pi(gkh) \,dk \qquad (g, h\in G).
\end{equation}
We recall that the projection operator
$$
P_\pi: \cH_\pi \to \cH_\pi^K , \quad v \mapsto \la v,\zeta_\pi\ra \zeta_\pi
$$
is given by the $K$-average
$$
P_\pi(v)
= \int_K \pi(k) v \,dk.
$$
The identity \eqref{mean} is now shown by the computation
\begin{align*}
\int_K \phi_\pi (gkh)\,dk &
= \int_K \la \pi(gkh)\zeta_\pi, \zeta_\pi\ra  dk
 = \left\la \int_K \pi(k) \pi(h)\zeta_{\pi} \,dk, \pi(g^{-1})\zeta_\pi \right\ra\\
&= \Big\la P_\pi (\pi(h)\zeta_\pi), \pi(g^{-1})\zeta_\pi\Big\ra
 = \Big\la \la\pi(h)\zeta_\pi, \zeta_\pi\ra \zeta_\pi, \pi(g^{-1})\zeta_\pi\Big\ra\\
&= \phi_\pi(g)\phi_\pi(h).
\end{align*}

\subsection{Abstract Plancherel Theorem for $L^2(Z)$}
We recall the abstract Plancherel decomposition, the Fourier transform and its inverse, often referred to as wave packet transform. Before we do that we give a short digression on the dual representation.

If $\pi\in\hat G$ is spherical then also its dual $\pi'$ is spherical. In fact, the elements of $\cH_{\pi}^{K}$ are fixed by the projection $P_{\pi}$.
Therefore, if $\cH_{\pi}^{K}\neq \{0\}$, then the dual projection $P_{\pi}':\cH_{\pi}'\to\cH_{\pi}' $ is non-zero and every element in its image is $K$-fixed. Therefore, $\pi'$ is spherical.

For $[\pi]\in\hat G$ we note that any two $G$-invariant inner products on $\cH_\pi$ are the same only up to multiplication by a positive scalar. However for the Hilbert tensor product $\cH_\pi \hat \otimes\cH_{\pi}'$ there is a canonical $G$-invariant inner product: $\cH_\pi \hat \otimes\cH_{\pi}'$ naturally identifies with the Hilbert space $\mathrm{HS}_{\pi}$ of Hilbert-Schmidt operators on $\cH_\pi$, which carries the canonical $G$-invariant inner product
$$
\mathrm{HS}_{\pi}\times\mathrm{HS}_{\pi}\to\C,\quad (A,B)\mapsto tr(A\circ B^{\dagger}).
$$
Below we will be interested in the one-dimensional space $\cM_{\pi}=(\cH_{\pi}')^K$ and
the subspace $\cH_\pi \otimes \cM_{\pi}\subset \cH_\pi\otimes \cH_{\pi}'$ equipped with the canonical inner product.
 Note that $\cM_\pi =(\cH_{\pi}')^{K}=(\cH_\pi^K)'$ and so
$\cH_\pi\otimes \cM_\pi = \Hom_\C (\cH_\pi^K, \cH_\pi)$.

We define the Fourier transform $\cF f(\pi)\in \Hom_\C (\cH_\pi^K, \cH_\pi)$ of a function $f \in C_c(Z)=C_c(G)^K$ at $\pi\in\hat{G}_{s}$ by
$$
\cF f(\pi)v= \pi(f)v\qquad (v\in \cH_{\pi}^{K}).
$$
We may interpret $\cF f(\pi)$ as an element in $\cH_{\pi}\otimes\cM_{\pi}\simeq \Hom_\C (\cH_\pi^K, \cH_\pi)$.

Harish-Chandra proved in \cite[Theorem 7]{HC_Type1} that $G$ is of type $I$ in the terminology of Murray and von Neumann, which makes the abstract Plancherel theory in \cite{Dixmier} applicable.
According to the abstract Plancherel theorem for $Z$ the families
$$
\cF f\in  \big(\cH_{\pi}\otimes \cM_{\pi}\big)_{\pi\in \hat{G}_{s}}\qquad \big(f\in C_{c}(Z)\big)
$$
are Borel measurable and there exists a unique positive Borel measure $\mu$ on $\hat{G}_{s}$, called the Plancherel measure of $Z$, so that the Fourier transform
$$
\cF:C_{c}(Z)\to \big(\cH_{\pi}\otimes \cM_{\pi}\big)_{\pi\in \hat{G}_{s}}
$$
extends to a $G$-equivariant unitary isomorphism
\begin{equation} \label{eq abstract Plancherel thm}
\cF: \big(L,L^2(Z)\big)\to \left(\int_{\hat G_s}^\oplus \pi \otimes \mathrm{id}_{\cM_{\pi}} \,d\mu(\pi),
    \int_{\hat G_s}^\oplus \cH_\pi \otimes \cM_{\pi} \,d\mu(\pi)\right).
\end{equation}

The inverse of $\cF$ is for compactly supported measurable sections $v: \pi \mapsto v_\pi\otimes \eta_{\pi}$ of $\big(\cH_{\pi}\otimes \cM_{\pi}\big)_{\pi\in \hat{G}_{s}}$ given by
$$
\cF^{-1}\big((v_\pi\otimes\eta_{\pi})_{\pi \in \hat G_s}\big)(gK) = \int_{\hat G_s} \eta_{\pi}\big(\pi(g^{-1})v_\pi\big) \,d\mu(\pi)\qquad (g\in G).
$$
For $K$-invariant functions $f\in C_c(Z)^{K}=C_{c}(G)^{K\times K}$ the Fourier transform $\cF(f)(\pi)$ at $\pi$ maps the one-dimensional space $\cH_{\pi}^{K}$ to itself, hence it is given by multiplication by a scalar $\cF_{\mathrm{sph}}f (\pi)$. The corresponding map
$$
\cF_{\mathrm{sph}}:C_{c}(G)^{K\times K}\to L^{2}(\hat{G}_{s},\mu)
$$
is called the spherical Fourier transform and is given by
$$
\cF_{\mathrm{sph}}f(\pi)=\int_G f(g) \phi_\pi(g) \,dg \qquad (f\in C_{c}(G)^{K\times K},\pi \in \hat G_s).
$$
The spherical Fourier transform extends to a unitary isomorphism
$$
\cF_{\mathrm{sph}}:L^{2}(G)^{K\times K}\to L^{2}(\hat{G}_{s},\mu).
$$
Its inverse is for compactly supported continuous functions $\psi$ on $\hat{G}_{s}$ given by
$$\cF_{\mathrm{sph}}^{-1}\psi (g) = \int_{\hat G_s} \psi(\pi) \phi_\pi(g^{-1})\,d\mu(\pi)\qquad(g\in G).$$

Next we wish to determine the support of the Plancherel measure.

\subsection{Irreducible spherical representations}
We assume that the reader is familiar with the basic theory of Harish-Chandra modules. The prototypical Harish-Chandra module is the space $V_{\pi}$ of $K$-finite vectors of some $\pi\in \hat G$, say
$V_\pi = \cH_\pi^{K-{\rm finite}}$. It was the discovery of Harish-Chandra that $V_\pi$ is an irreducible module for $\gf$.

Let us recall the (minimal) principal series representations of $G$. Let $\sigma\in\hat{M}$ and let $\cH_\sigma$ be a finite dimensional model Hilbert space for $\sigma$. Recall from (\ref{eq Def rho}) that $\rho$ is the half-sum of all positive roots with respect to $P$ counted with multiplicity. For $\lambda\in\af_\C^*$ we define a representation of $\overline{P}$ on $\cH_\sigma$ by
$$
\sigma_\lambda(m a\overline{n}) = \sigma(m) a^{\lambda - \rho} \qquad (ma\overline{n}\in \overline{P} = MA \overline{N}).
$$
The smooth principal series representation with parameters $(\sigma,\lambda)$ is the smooth representation given by the left-regular action of $G$ on the Fr{\'e}chet space
\begin{equation} \label{def princ}
\cH_{\sigma, \lambda}^\infty= \{ f \in C^\infty (G, \cH_\sigma)\mid f(g\overline{p}) = \sigma_\lambda(\overline{p})^{-1}  f (g)\}.
\end{equation}
We denote by $V_{\sigma, \lambda}$ its Harish-Chandra module of $K$-finite vectors.
Restriction to $K$ yields an isomorphism of $K$-modules
$$
\cH_{\sigma, \lambda}^\infty\simeq C^\infty(K\times_M \cH_\sigma)
:=\{ f \in C^\infty(K, \cH_\sigma) \mid f(km) = \sigma_\lambda(m)^{-1} f (k), k\in K, m\in M\}.
$$
Integration over $K$ yields the non-degenerate $G$-invariant bilinear pairing
\begin{equation} \label{pairing}
\cH_{\sigma, \lambda}^\infty  \times \cH_{\sigma', -\lambda}^\infty, \quad (f_1, f_2)\mapsto
\int_K \big(f_1(k), f_2(k)\big)_\sigma \,dk,
\end{equation}
where $(\cdot, \cdot)_\sigma$ denotes the natural pairing of $\cH_\sigma$ and its dual $\cH_{\sigma'} = \cH_{\sigma}'$. We define
$$
\cH_{\sigma, \lambda}
:=\big\{f: G\to \cH_{\sigma}\mid f(g\overline{p})=\sigma_{\lambda}(\overline{p})^{-1}f(g)\text{ and }f|_{K}\in L^2(K\times_M \cH_\sigma)\big\}.
$$
and equip it with the Hilbert space structure given by the inner product on $L^2(K\times_M \cH_\sigma)$. Note that $\cH_{\sigma,\lambda}$ and  $L^2(K\times_M \cH_\sigma)$ are isomorphic as unitary representations of $K$. The left regular representation $\pi_{\sigma,\lambda}$ of $G$ on $\cH_{\sigma, \lambda}$ is called a representation of the principal series. It is unitary if and only if  $\lambda\in i\af^*$. The Harish-Chandra module $V_{\sigma,\lambda}$ of $\pi_{\sigma,\lambda}$ is the space consisting of all $f\in \cH_{\sigma,\lambda}^{\infty}$ for which the restriction $f|_{K}$ is $K$-finite, i.e. $f|_{K}$ is a regular section of $K\times_{M}\cH_{\sigma}\to K/M$.

We now state the Casselman subrepresentation theorem.

\begin{theorem}\label{Thm Casselman subrep thm}
Every irreducible Harish-Chandra module $V$ embeds into the Harish-Chandra module of a principal
series representation, i.e. there exists an injective $(\gf,K)$-morphism
$V\hookrightarrow V_{\sigma, \lambda}$ for some pair $(\sigma,\lambda)\in\hat{M}\times\fa_{\C}^{*}$.
\end{theorem}

The theorem is a consequence of the non-vanishing of the zeroth $\overline{\nf}$-homology
of $V$, i.e. $H_0(V, \overline{\nf} ) = V/ \overline{\nf} V\neq \{0\}$, and Frobenius reciprocity for $(\gf, K)$-modules. See \cite[Theorem 3.8.3 \& Lemma 3.8.2]{W}.

We call a Harish-Chandra module $V$ spherical provided that $V^K\neq \{0\}$. If $\sigma=\1$, then we abbreviate and write $\pi_\lambda$, $V_{\lambda}$, $\cH_{\lambda}^{\infty}$ and $\cH_{\lambda}$ instead of $\pi_{\1, \lambda}$, $V_{\1,\lambda}$, $\cH_{\1,\lambda}^{\infty}$ and $\cH_{\1,\lambda}$. Note that $V_\lambda^K \simeq\C \1_{K/M} \subset \C[K\times_M V_\1]\simeq V_\lambda$, where $\1_{K/M}$ is the constant function $1$ on $K/M$.

Theorem \ref{Thm Casselman subrep thm} has the following corollary.

\begin{cor} \label{cor Cass}
Let $V$ be an irreducible spherical Harish-Chandra module. Then there exists an injective $(\gf,K)$-morphism $V\hookrightarrow V_\lambda$ for some $\lambda\in\fa_{\C}^{*}$.
\end{cor}

\begin{proof}
Let $\sigma\in \hat{M}$ and $\lambda\in\fa_{\C}^{*}$ be so that there exists an injective $(\gf,K)$-morphism $V\hookrightarrow V_{\sigma,\lambda}$. The image of a non-zero $K$-fixed vector of $V$ will be $K$-fixed in $V_{\sigma, \lambda}\simeq \C[K\times_M V_\sigma]$. But $\C[K\times_M V_\sigma]^K\neq \{0\}$ if and only if $\sigma=\1$.
\end{proof}

We need the following result of Kostant \cite[Theorem 1]{Kos}.

\begin{prop}\label{Prop irreducibility}
The Harish-Chandra module $V_\lambda$ is irreducible for every $\lambda\in  i \af^*$.
\end{prop}

From the proposition we obtain a map
$$
i \af^* \to \hat G_s, \quad \lambda\mapsto [\pi_\lambda].
$$
Harish-Chandra has shown that the fibers of this map are the Weyl group orbits, i.e.
\begin{equation}\label{W inv}  [\pi_\lambda]=[\pi_{\lambda'}]\quad\hbox{if and only if}\quad   \lambda \in W\lambda'.\end{equation}
We will sketch a proof of \eqref{W inv} later in Remark \ref{rmk Symmetry spherical functions}.
We thus obtain an inclusion $i\af^*/W\hookrightarrow \hat G_s$.
In Theorem \ref{thm planch} it will be shown that
$\supp \mu = i\af^*/W$ and that $\mu$ is in the measure class of the Lebesgue measure.

For many applications the space  $\cH_{\lambda}$ is too small. We therefore introduce the space of so-called generalized vectors. We recall that a Gelfand triple consists of a Hilbert space $\cH$, a Fr\'echet space $E$ and a dual Fr\'echet space $F$ such that $E\subset \cH \subset F$ with $F'\simeq E$ and $E'\simeq F$. In our context
$\cH=\cH_\lambda$ and $E=\cH_\lambda^\infty$. We define the space of generalized vectors $\cH_\lambda^{-\infty}:=(\cH_{-\lambda}^\infty)'$. Now from $\cH_\lambda'\simeq \cH_{-\lambda}$ we obtain $\cH_\lambda$ is naturally contained in $ \cH_{\lambda}^{-\infty}$.
As $K$-representation $\cH_{\lambda}^{-\infty}$ is isomorphic to the space $\cD'(K/M)$ of distributions on $K/M$. In view of \cite[Chapitre III, Th{\'e}or{\`e}me XIV]{Schwartz} the dual of the latter is $C^{\infty}(K/M)$, which is isomorphic to $\cH_{\lambda}^{\infty}$ as $K$-representation. Therefore, the triple
$$ \cH_\lambda^\infty \subset \cH_\lambda\subset \cH_{\lambda}^{-\infty}$$
is a Gelfand triple with $G$-equivariant inclusions. Important for us is the mollifying property
$$C_c^\infty(G)\times \cH_{\lambda}^{-\infty} \to \cH_\lambda^\infty,
\quad (f,\xi)\mapsto \pi_{\lambda}(f)\xi:=\int_G f(g)\pi_{\lambda}(g)\xi \,dg, $$
where $\pi_{\lambda}(g)\xi=\xi\circ \pi_{-\lambda}(g^{-1})$ is the usual dual action.

We end this section with choosing a specific $K$-fixed vector in $\cH_{\lambda}$, which will be used in the remainder of this text.
For every $\lambda\in \fa_{\C}^{*}$ we define $\eta_\lambda\in \cH_{\lambda}$ to be the element so that
$$
\eta_{\lambda}(k)=1\qquad (k\in K).
$$
Then $\cH_{\lambda}^{K}=\C \eta_{\lambda}$.
We recall the Iwasawa projection $\acomp : G \to A$ with respect to the Iwasawa decomposition $G=KA\overline{N}$. We have
\begin{equation}\label{eq eta_lambda}
\eta_{\lambda}(g)=\acomp(g)^{\rho-\lambda}\qquad \big(\lambda\in \fa_{\C}^{*}, g\in G\big).
\end{equation}

\subsection{Harish-Chandra's spherical functions}
We abbreviate the notation for spherical functions and set $\phi_\lambda= \phi_{\pi_\lambda}$ whenever $\lambda\in  i\af^*$. In view of (\ref{eq eta_lambda}) we have
\begin{equation}\label{phi lambda}
\phi_\lambda(g)  = \la \pi_\lambda(g)\eta_\lambda, \eta_\lambda\ra
  = \int_K \acomp(g^{-1}k)^{\rho -\lambda} \,dk \qquad(\lambda\in i\fa^{*}, g\in G).
 \end{equation}
We can define spherical functions $\phi_{\lambda}$ for all $\lambda\in \af_\C^*$ by using the bilinear pairing of $V_\lambda$ and $V_{-\lambda}$ and taking the matrix coefficient of the two normalized spherical vectors, i.e.
$$\phi_\lambda(g) := (\pi_\lambda(g) \eta_\lambda, \eta_{-\lambda})\qquad(g\in G).$$
Analytic continuation of \eqref{phi lambda} then yields Harish-Chandra's integral formula
$$
\phi_\lambda(g)=\int_K \acomp(g^{-1}k)^{\rho -\lambda} \,dk \qquad (\lambda\in \af_\C^*, g\in G)
$$
in general.

\subsection{$Z$-tempered representations}
In \cite{B} Bernstein developed from first principles a general theory of $Z$-tempered representations for  homogeneous spaces $Z=G/H$. We will briefly recall this for our special situation of $Z=G/K$. See also the summary in \cite{KS}. A crucial object is the volume weight function on $Z$ which is defined as follows. Let $B\subset G$ be a compact neighborhood of $\1$. Then we define a volume weight of $Z$ as
$$ {\bf v}(z) = \vol_Z(Bz) \qquad (z \in Z).$$
The choice  of $B$ is irrelevant as another choice yields a comparable weight, i.e. the two weights are bounded against each other by positive constants. It is an interesting exercise to verify that
$${\bf v} (ka\cdot z_0)\asymp a^{-2\rho} \qquad (k \in K, a\in A^-).$$
See \cite[Proposition 4.3]{KKSS} for a more general result in the context of real spherical spaces. The radial function
$${\bf w}:Z\to \R_{\geq 1},\quad ka\cdot z_{0}\mapsto 1+ \|\log (a)\|$$
will also be of use.

Finally for $\pi\in \hat G_s$, $v\in \cH_\pi$ and $\xi\in \cH_{\pi}'$, we define the matrix coefficient
$$m_{v, \xi}:G\to\C,\quad g\mapsto \xi\big(\pi(g^{-1})v\big).$$
The following is a special case of Bernstein's theorem, see the analytic necessary condition in \cite[Section 0.2]{B}.

\begin{prop}\label{Prop tempered estimate}
There exists a number $r \in \R$ such that for almost all $\pi \in \supp\mu$  and $\eta\in (\cH_{\pi}')^{K}$ with $\|\eta\|=1$ we have
\begin{equation}  \label{tempered}
q_{\pi,r}(v):=\sup_{z\in Z} |m_{v, \eta}(z)| \sqrt{{\bf v}(z)} {\bf w}(z)^r <\infty \qquad
( v \in \cH_\pi^\infty).\end{equation}
\end{prop}
We call $\pi\in \hat G_s$ {\it tempered} provided that the estimate \eqref{tempered} holds true for some $r\in \R$.

\begin{theorem}\label{Thm classification of tempered reps}
An irreducible spherical representation $\pi$ of $G$ is tempered if and only if $[\pi]=[\pi_\lambda]$ for some $\lambda\in i\af^*$.
\end{theorem}

\begin{proof}
We provide here the proof for the temperedness of all representations $\pi_{\lambda}$ with $\lambda\in i\fa^{*}$. The proof for the converse implication will be given in section \ref{Subsection Temperedness revisited}.

Recall the Iwasawa-projections $\kcomp:G\to K$ and $\acomp:G\to A$ along the Iwasawa decomposition $G=KA\overline{N}$.
Let $\lambda\in i\fa^{*}$. For every $v\in \cH_{-\lambda}^{\infty}$ we have
\begin{align*}
|m_{v,\eta_{\lambda}}(gK)|
&=\big|\int_{K}v(gk)\,dk\big|
=\big|\int_{K}\acomp(gk)^{\rho+\lambda}v\big(\kcomp(gk)\big)\,dk\big|\\
&\leq\sup_{k\in K}|v(k)|\, \int_{K}|\acomp(gl)^{\rho+\lambda}|\,dl
=\sup_{k\in K}|v(k)|\, \phi_{0}(g^{-1}).
\end{align*}
It thus suffices to show that
$$
\phi_{0}(z)
\leq C {\bf v}(z)^{-\frac{1}{2}} {\bf w}(z)^{-r}
\qquad (z\in Z)
$$
for some $C>0$ and $r\in \R$. The latter bound follows from Harish-Chandra's estimate of the spherical function $\phi_{0}$, see \cite[Theorem 4.6.4]{GangolliVaradarajan}.
\end{proof}

\begin{cor} \label{cor support}
The map
$$
\iota:i\af^*/W\to\hat{G}_{s},\quad\lambda\mapsto [\pi_{\lambda}]
$$
is a continuous injection onto the subset $\hat{G}_{s,\mathrm{temp}}\subseteq\hat G_s$ of irreducible tempered spherical representations.
\end{cor}

\section{Intertwiners and asymptotics}\label{Section Intertwiners and Asymptotics}
\subsection{Heuristic from finite dimensional representations}
The aim of this section is to explain the ideas behind the asymptotics of $K$-fixed generalized vectors.

Let $V$ be a finite dimensional irreducible representation and $\tilde V$ its dual. We assume that $\tilde V^K\neq \{0\}$ and let $0\neq \eta \in \tilde{V}^K$. Let $\lambda\in \af^*$ be the lowest weight of $\tilde V$ and
\begin{equation} \label{eta weights}
\eta=\sum_{\mu \in \lambda+\N_0[\Sigma^+]}\eta^\mu
\end{equation}
the expansion of $\eta$ in $\af$-weight vectors.

We claim that $\eta^{\lambda}\neq 0$.
Let $v_{-\lambda} \in V^{N}\setminus\{0\}$ be a highest weight vector with highest weight $-\lambda$.
First note that $\eta^\lambda(v_{-\lambda})=\eta(v_{-\lambda})$. It thus suffices to show that $\eta(v_{-\lambda})\neq 0$.
In fact $V = \U(\gf)v_{-\lambda}$ as $V$ is irreducible and $\U(\gf)=\U(\kf)\U(\af+\nf)$ by the theorem of Poincar{\'e}, Birkhoff and Witt. Therefore,  $V = \U(\kf)v_{-\lambda}$. On the other hand $\U(\kf)\eta=\C\eta$. If $\eta(v_{-\lambda})$ would be $0$, then $\eta(V)=\eta(\U(\gf)v_{-\lambda})=\{0\}$, hence $\eta=0$, which is a contradiction.

In fact the argument just given shows more: since $\eta(v_{-\lambda})\neq 0$ for any $\eta\in \tilde{V}^{K}\setminus\{0\}$ and any $v_{-\lambda}\in V^{N}\setminus\{0\}$, it follows that $\tilde{V}^{K}$ and $V^{N}$ are both $1$-dimensional. In particular we obtain an alternative proof for Gelfand's Theorem \ref{thm Gelfand} for irreducible finite dimensional representations.
Moreover, for every $m\in M$ and $v_{-\lambda}\in V^{N}$ we have
$$
\eta^{\lambda}(m\cdot v_{-\lambda})
=\eta(m\cdot v_{-\lambda})
=(m^{-1}\cdot \eta)(v_{-\lambda})
=\eta(v_{-\lambda})
=\eta^{\lambda}(v_{-\lambda}),
$$
and hence $V^{N}=V^{MN}$.
By dualizing we find that $\tilde{V}^{\overline{N}}=\tilde{V}^{M\overline{N}}= \C\eta^{\lambda}$ is $1$-dimensional.

It follows from (\ref{eta weights}) that we can obtain $\eta^{\lambda}$ from $\eta$ as a limit
\begin{equation} \label{limit eta}
\eta^{\lambda}
=\lim_{t\to \infty} e^{-t\lambda(X)} (\exp(tX)\cdot \eta)
\end{equation}
where $X$ is any element in $\af^{--}$.
We wish to interpret \eqref{limit eta} in terms of matrix coefficients. Let
$$m_{v,\eta}(g):=\eta(g^{-1}v) \quad \hbox{and} \quad m_{v, \eta^{\lambda}}(g) =\eta^{\lambda}(g^{-1}v)\qquad (g\in G).$$
Now $m_{v,\eta}$ descends to a function on $Z=G/K$ whereas $m_{v,\eta^{\lambda}}$ descends to a function on $G/M\overline{N}$. The natural comparison of $m_{v,\eta}$ with $m_{v,\eta^{\lambda}}$ is thus on $A$, where $m_{v,\eta^{\lambda}}(a)= a^\lambda \eta^{\lambda}(v)$ behaves like a character.
For every $X\in\fa$ and $t\in \R$
$$
m_{v,\eta}\big(\exp(tX)\big)=m_{v,\eta_{{\lambda}}}\big(\exp(tX)\big)+R_{t}(v)
$$
where
$$
R_{t}(v)
=\sum_{\mu\in \big(\lambda+\N_{0}[\Sigma^{+}]\big)\setminus\{\lambda\}}m_{v,\eta^{\mu}}\big(\exp(tX)\big).
$$
For $X\in \fa^{--}$ the remainder $R_{t}(v)$ is bounded by
$$
|R_{t}(v)|\leq c e^{t\lambda(X)+t\epsilon\rho(X)}\|v\|\|\eta\|\qquad \big(v\in V, t\geq0\big),
$$
where
$$
c=\#\big\{\mu\in\lambda+\N_{0}[\Sigma^{+}]\mid \eta^{\mu}\neq 0\big\}
$$
and
$$
\epsilon
=\min\Big\{\frac{\mu(X)-\lambda(X)}{\rho(X)}\mid \mu\in  \big(\lambda+\N_{0}[\Sigma^{+}]\big)\setminus\{\lambda\}\Big\}
\geq \min\Big\{\frac{\alpha(X)}{\rho(X)}\mid \alpha\in\Sigma^{+}\Big\}>0.
$$
A more quantitative version of \eqref{limit eta} is then
\begin{equation}\label{remainder fd}
|m_{v,\eta}\big(\exp(tX)\big)- m_{v,\eta^{\lambda}} \big(\exp(tX)\big)| \leq c e^{t\lambda(X) +t\e \rho(X)} \|v\| \|\eta\|\qquad \big(v \in V,t\geq 0\big)
\end{equation}
for an $\e>0$ which depends only on the distance of $\frac{X}{\|X\|}$ to the walls of $\af^-$ and not on $V$.

Our aim is to develop a theory for asymptotics for matrix coefficients of spherical representations. We do this in two ways: principal asymptotics (Theorem \ref{Thm Principal asymptotics}) and the constant term approximation (Theorem \ref{thm const}). In each case we determine an analogue of (\ref{remainder fd}).

\subsection{Intertwining operators}
In this section we introduce the theory of intertwining operators on principal series representations, which was developed by Knapp and Stein in \cite{KnappStein}. The intertwiners will be used in the next section to describe the asymptotics of $K$-fixed generalized vectors.

We first need to fix normalizations of Haar measures on certain subgroups of conjugates of $N$.
Let $Q$ be any minimal parabolic subgroup containing $A$. Then $Q=MAN_{Q}$, where $N_{Q}$ is the nilpotent radical of $Q$. The map
$$
(N_{Q}\cap N)\times (N_{Q}\cap\overline{N})\to N_{Q},\quad (n_{1},n_{2})\mapsto n_{1}n_{2}
$$
is a diffeomorphism.
The Haar measure on $N_{Q}$ is left $(N_{Q}\cap N)$- and right $(N_{Q}\cap \overline{N})$-invariant, and hence we may and will normalize the Haar measures on $N_{Q}\cap N$ and $N_{Q}\cap \overline{N}$ so that
\begin{equation}\label{eq Decomp of N integral}
\int_{N_{Q}}\phi(n)\,dn
=\int_{N_{Q}\cap N}\int_{N_{Q}\cap \overline{N}}\phi(n_{1}n_{2})\,dn_{2}\,dn_{1}\qquad \big(\phi\in C_{c}(N_{Q})\big).
\end{equation}
We use (\ref{eq Decomp of N integral}) to fix the normalization of the invariant measure on the quotient $N_{Q}/(N_{Q}\cap\overline{N})$ by demanding that
\begin{equation}\label{eq normalization of measure on quotient}
\int_{N_{Q}/(N_{Q}\cap\overline{N})}\psi(\overline{n})\,d\overline{n}
=\int_{N_{Q}\cap N}\psi(n)\,dn
\qquad \Big(\psi\in C_{c}\big(N_{Q}/(N_{Q}\cap\overline{N})\big)\Big).
\end{equation}

Let $w\in W$ and let $\omega\in N_{K}(\fa)$ be a representative of $w$.
We define
$$
{}^{w}\!\overline{N}:=\omega^{-1}\overline{N}\omega.
$$
By applying the normalization from (\ref{eq normalization of measure on quotient}) to $Q=\omega^{-1}\overline{P}\omega$, we obtain that the invariant measure on  the quotient ${}^{w}\!\overline{N}/({}^{w}\!\overline{N}\cap\overline{N})$ is given by
$$
\int_{{}^{w}\!\overline{N}/({}^{w}\!\overline{N}\cap\overline{N})}\psi(\overline{n})\,d\overline{n}
=\int_{{}^{w}\!\overline{N}\cap N}\psi(n)\,dn
\qquad \Big(\psi\in C_{c}\big({}^{w}\!\overline{N}/({}^{w}\!\overline{N}\cap\overline{N})\big)\Big).
$$

We define $\partial({}^{w}\!\overline{N}\,\overline{P}):=\mathrm{cl}({}^{w}\!\overline{N}\,\overline{P})\setminus {}^{w}\!\overline{N}\,\overline{P}\subseteq G/\overline{P}$.
Let $\lambda\in \fa_{\C}^{*}$. For $v\in \cH_{\lambda}^{\infty}$ we define
$$
\Omega_{w,v}:=\big\{g\in G\mid\supp(v)\cap g\omega\partial({}^{w}\!\overline{N}\,\overline{P})=\emptyset\big\}.
$$
Note that $\partial({}^{w}\!\overline{N}\,\overline{P})$ is closed as it is a finite union of  closures of ${}^{w}\!\overline{N}$ orbits in $G/\overline{P}$ of dimension strictly smaller than the dimension of ${}^{w}\!\overline{N}\,\overline{P}$. Therefore, $\Omega_{w,v}$ is an open right $\overline{P}$-invariant subset of $G$. We may now define
\begin{equation}\label{eq Def J_w(lambda)}
J_{w}(\lambda)v:\Omega_{w,v}\to\C,\quad g\mapsto \int_{{}^{w}\!\overline{N}\cap N}v(g\omega n)\,dn
= \int_{{}^{w}\!\overline{N}/({}^{w}\!\overline{N}\cap \overline{N})}v(g\omega\overline{n})\,d\overline{n}.
\end{equation}
The integrals are absolutely convergent as ${}^{w}\!\overline{N}\cap N\ni n\mapsto v(g\omega n)$ is by definition compactly supported for all $g\in \Omega_{w,v}$.

Note that $J_{w}(\lambda)v$ is right $\overline{N}$-invariant. Moreover, for all $g\in \Omega_{w,v}$, $m\in M$, $a\in A$ and $\overline{n}\in \overline{N}$
\begin{align*}
J_{w}(\lambda)v(gma\overline{n})
&=J_{w}(\lambda)v(gma)
= \int_{{}^{w}\!\overline{N}\cap N}v(gma\omega n)\,dn\\
&=a^{-w\lambda+w\rho}\int_{{}^{w}\!\overline{N}\cap N}v\big(g\omega(\omega^{-1} ma\omega)n(\omega^{-1} ma\omega)^{-1}\big)\,dn\\
&=a^{-w\lambda+w\rho}a^{w(w^{-1}\rho-\rho)}\int_{{}^{w}\!\overline{N}\cap N}v(g \omega n)\,dn
=a^{-w\lambda+\rho}J_{w}(\lambda)v(g).
\end{align*}
If $\lambda\in \fa_{\C}^{*}$ satisfies
\begin{equation}\label{eq lambda condition}
\Re\langle \lambda,\alpha\rangle>0\qquad \big(\alpha\in -\Sigma^{+}\cap w^{-1}\Sigma^{+}\big)
\end{equation}
then the integrals
$$
 \int_{{}^{w}\!\overline{N}\cap N}v(g\omega n)\,dn
$$
are absolutely convergent for every $v\in \cH_{\lambda}^{\infty}$ and $g\in G$. See \cite[Proposition 7.8]{Knapp}.
Therefore, for $\lambda\in \fa_{\C}^{*}$ satisfying (\ref{eq lambda condition}) the functions $J_{w}(\lambda)v$  with $v\in \cH_{\lambda}^{\infty}$ extend to all of $G$, are given by convergent integrals and belong to $\cH_{w\lambda}^{\infty}$.
We denote these extensions again by $J_{w}(\lambda)v$. Note that $J_{w}(\lambda)$ is equivariant, and hence intertwines the representations $\cH_{\lambda}^{\infty}$ and $\cH_{w\lambda}^{\infty}$.

When viewed as a map $C^{\infty}(K/M)\to C^{\infty}(K/M)$ the dependence of $J_{w}(\lambda)$ on $\lambda$ is holomorphic for $\lambda\in\fa_{\C}^{*}$ satisfying (\ref{eq lambda condition}). The family of operators $\lambda\mapsto J_{w}(\lambda)\in \End\big(C^{\infty}(K/M)\big)$ extends to a meromorphic family on $\fa_{\C}^{*}$.
See \cite[Corollary 7.13]{Knapp}.
Furthermore, $J_{w}(\lambda)v(g)$ is given by (\ref{eq Def J_w(lambda)}) for all $v\in \cH_{\lambda}^{\infty}$ and $g\in\Omega_{w,v}$.  The operators $J_{w}(\lambda)$ are equivariant and are called standard intertwining operators.

Since $J_{w}(\lambda)$ is an equivariant operator, there exists a unique meromorphic function $\cfunc_{w}:\fa_{\C}^{*}\to \C$ so that
$$
J_{w}(\lambda)\eta_{\lambda}
=\cfunc_{w}(\lambda)\eta_{w\lambda}
$$
as a meromorphic identity.
The function $\cfunc_{w}$ is for $\lambda\in \fa_{\C}^{*}$ satisfying (\ref{eq lambda condition}) given by the convergent integral
$$
\cfunc_{w}(\lambda)
=\big(J_{w}(\lambda)\eta_{\lambda}\big)(\1)
=\int_{{}^{w}\!\overline{N}\cap N}\acomp(n)^{\rho-\lambda}\,dn.
$$
The Gindikin-Karpelevich formula in \cite{GindikinKarpelevich} (see also the article \cite{BhanuMurthy} by Bhanu Murthy) exhibits the $\cfunc$-functions as a product of $\cfunc$-function for groups of real rank $1$. The latter are explicitly given in terms of quotients of $\Gamma$-functions.
See \cite[Theorem 4.7.5]{GangolliVaradarajan}.
It follows from the  Gindikin-Karpelevich formula that $\cfunc_{w}(\lambda)\neq 0$ for any $\lambda\in \fa_{\C}^{*}$.

The follow lemma determines the dual operator to $J_{w}(\lambda)$.

\begin{lemma}\label{Lemma dual of intertwiner}
For all $w\in W$ and $\lambda\in \fa_{\C}^{*}$ so that $J_{w}$ does not have a pole at $\lambda$ we have
$$
\int_{K}\big(J_{w}(\lambda) u\big)(k) v(k)\,dk
=\int_{K} u(k)\big(J_{w^{-1}}(-w\lambda)v\big)(k)\,dk
\qquad (u\in \cH_{\lambda}, v\in \cH_{-w\lambda}).
$$
In particular, the dual operator of $J_{w}(\lambda)$ is equal to $J_{w^{-1}}(-w\lambda)$.
\end{lemma}

\begin{proof}
We recall that ${}^{w}\! \overline{N}=\omega^{-1}\overline{N}\omega$. Similarly we define
$$
{}^{w}\! N:=\omega^{-1}N\omega,\quad N^{w}:=\omega N\omega^{-1}\quad\text{and}\quad \overline{N}^{w}:=\omega\overline{N}\omega^{-1}.
$$

First let $\lambda\in \fa_{\C}^{*}$ satisfy (\ref{eq lambda condition}).
We then have
\begin{align*}
&\int_{K}\big(J_{w}(\lambda)u\big)(k)v(k)\,dk
=\int_{K}\int_{{}^{w}\!\overline{N}\cap N}u(k\omega n)\,dn\,v(k)\,dk\\
&\qquad=\int_{K}\int_{\overline{N}\cap N^{w}}u(kn\omega)v(k)\,dn\,dk
=\int_{K}\int_{\overline{N}\cap N^{w}}u(kn\omega)v(kn)\,dn\,dk
\end{align*}
Since 
$$
u(gan\omega)v(gan)=a^{\rho+w\rho}u(g\omega)v(g)\qquad(g\in G, a\in A, n\in \overline{N}\cap \overline{N}^{w})
$$
there exists a smooth function $\phi:G\to\C$ so that
$$
u(g\omega)v(g)=\int_{\overline{N}\cap \overline{N}^{w}}\int_{A}\phi(g\overline{n}a)\,da\,d\overline{n}\qquad(g\in G)
$$
with absolutely convergent integrals.
Using (\ref{eq Decomp of N integral}) and (\ref{eq Integral formula Iwasawa decomp}) we obtain
\begin{align*}
&\int_{K}\big(J_{w}(\lambda)u\big)(k)v(k)\,dk
=\int_{K}\int_{\overline{N}\cap N^{w}}\int_{\overline{N}\cap \overline{N}^{w}}\int_{A}\phi(kn\overline{n}a)\,da\,d\overline{n}\,dn\,dk\\
&\qquad=\int_{K}\int_{\overline{N}}\int_{A}\phi(k\overline{n}a)\,da\,d\overline{n}\,dk
=\int_{G}\phi(g)\,dg.
\end{align*}
We now use the right invariance of the Haar measure on $G$ and  again (\ref{eq Decomp of N integral}) and (\ref{eq Integral formula Iwasawa decomp}). This yields
\begin{align*}
&\int_{G}\phi(g)\,dg
=\int_{G}\phi(g\omega^{-1})\,dg
=\int_{K}\int_{\overline{N}}\int_{A}\phi(k\overline{n}a\omega^{-1})\,da\,d\overline{n}\,dk\\
&\qquad=\int_{K}\int_{\overline{N}\cap{}^{w}\! N}\int_{\overline{N}\cap{}^{w}\!\overline{N}} \int_{A}\phi(kn\overline{n}a\omega^{-1})\,da\,d\overline{n}\,dn\,dk\\
&\qquad=\int_{K}\int_{\overline{N}\cap{}^{w}\! N}\int_{\overline{N}^{w}\cap\overline{N}} \int_{A}\phi(kn\omega^{-1}\overline{n}a)\,da\,d\overline{n}\,dn\,dk\\
&=\int_{K}\int_{\overline{N}\cap{}^{w}\! N}u(kn)v(kn\omega^{-1})\,dn\,dk
=\int_{K}u(k)\big(J_{w^{-1}}(-w\lambda)v\big)(k)\,dk.
\end{align*}
This proves the assertion for $\lambda\in \fa_{\C}^{*}$ satisfying (\ref{eq lambda condition}). The identity follows for other $\lambda\in\fa_{\C}^{*}$ by meromorphic continuation.
\end{proof}

An important corollary of Lemma \ref{Lemma dual of intertwiner} is that the standard intertwining operator $J_{w}(\lambda)$ uniquely extends to a continuous operator
$$
J_{w}(\lambda):\cH_\lambda^{-\infty}\to \cH_{w\lambda}^{-\infty}
$$
given by
\begin{equation}\label{eq Extension of intertwiner to generalized vectors}
J_{w}(\lambda)\eta:=\eta\circ J_{w^{-1}}(-w\lambda)\qquad (\eta\in\cH_{\lambda}^{-\infty}).
\end{equation}
When the intertwining operators are viewed as maps $\cD'(K/M)\to\cD'(K/M)$, they form a meromorphic family with family parameter $\lambda\in\fa_{\C}^{*}$. This family is the unique meromorphic family of $G$-equivariant linear maps, which map $\eta_{\lambda}$ to $\cfunc_{w}(\lambda)\eta_{w\lambda}$.

We now normalize the operators $J_{w}(\lambda)$ by setting
\begin{equation}\label{eq Def I}
I_{w}(\lambda):=\frac{1}{\cfunc_{w}(\lambda)}J_{w}(\lambda):\cH_\lambda^{-\infty}\to \cH_{w\lambda}^{-\infty}
\qquad(\lambda\in \fa_{\C}^{*}).
\end{equation}
Note that $I_{w}(\lambda)$ maps $\eta_\lambda$ to $\eta_{w\lambda}$. The operators $I_{w}(\lambda)$ are called normalized standard intertwining operators.
When the normalized intertwining operators are viewed as maps $\cD'(K/M)\to\cD'(K/M)$, they form a meromorphic family with family parameter $\lambda\in\fa_{\C}^{*}$. This family is the unique meromorphic family of $G$-equivariant linear maps, which map $\eta_{\lambda}$ to $\eta_{w\lambda}$ for all $\lambda$.  The uniqueness of the family implies that
$$
I_{w_{2}}(w_{1}\lambda)\circ  I_{w_{1}}(\lambda)=  I_{w_{2}w_{1}}(\lambda) \qquad (w_{1},w_{2}\in W).
$$
In particular
$$ I_w(\lambda)^{-1}=  I_{w^{-1}}(w\lambda) \qquad (w\in W).$$

If $\lambda\in i\fa^{*}$, then $\pi_{\lambda}$ is unitary and irreducible. We may therefore apply Schur's lemma from \cite[Lemma 0.5.2]{W}
and obtain that up to multiplication by a constant, the adjoint $I_{w}(\lambda)^{\dagger}$ and the inverse of $I_{w}(\lambda)$ coincide.
Since
$$
\langle \eta_{\lambda}, I_{w}(\lambda)^{\dagger}I_{w}(\lambda) \eta_{\lambda}\rangle
=\langle I_{w}(\lambda) \eta_{\lambda}, I_{w}(\lambda) \eta_{\lambda}\rangle
=\langle \eta_{\lambda},\eta_{\lambda}\rangle
=1
$$
it follows that in fact $I_{w}(\lambda)$ is unitary for all $\lambda\in i\fa^{*}$.
The meromorphic family $\lambda\mapsto I_{w}(\lambda)$ is therefore holomorphic on an open neighborhood of $i\fa^{*}$. See also \cite[Section XIV, \S 6]{Knapp}.

\begin{rmk} \label{rmk Symmetry spherical functions} Since the normalized intertwining operators $I_{w}(\lambda)$ are isomorphisms for $\lambda\in i\fa^{*}$ and $w\in W$ we have $[\pi_\lambda]=[\pi_{w\lambda}]$. This implies by analytic continuation that
$\phi_\lambda=\phi_{w\lambda}$ for all $w\in W$ and $\lambda\in\af_{\C}^{*}$.

If $\lambda,\lambda'\in \fa_{\C}^{*}$ then $\pi_{\lambda}$ and $\pi_{\lambda'}$ have equal infinitesimal character if and only if $\lambda\in W\lambda'$. See \cite[Propositions 8.20 \& 8.22]{Knapp}. Therefore, $[\pi_\lambda]=[\pi_{\lambda'}]$ for $\lambda, \lambda'\in i\fa^{*}$ if an only if $\lambda\in W\lambda'$.
\end{rmk}

\subsection{Principal asymptotics of $K$-fixed vectors}
For $\lambda\in \fa_{\C}^{*}$, $\xi\in \cH_{\lambda}^{-\infty}$ and $v\in \cH_{-\lambda}^{\infty}$ we define the matrix coefficient
$$
m_{v,\xi}:G\to\C,\quad g\mapsto \xi\big(\pi_{-\lambda}(g^{-1})v\big).
$$
In this section we investigate the asymptotic behavior along geodesics through the origin $eK\in Z$ of matrix coefficients $m_{v,\eta_{\lambda}}$ with $v\in \cH_{-\lambda}^{\infty}$. The approach is motivated by  \cite[Lemma 3.12]{Langlands}. See \cite[Section 5]{KKOS} for a more general result for real spherical subgroups.

\begin{theorem}\label{Thm Principal asymptotics}
Let $\lambda\in\fa_{\C}^{*}$ and $X\in\fa^{--}$.  Let $w_{0}\in W$ be the longest element and let $\omega_{0}\in N_{K}(\fa)$ be a representative for $w_{0}$. If $v\in \cH_{-\lambda}^{\infty}$ satisfies $\supp(v)\subset \omega_{0}N\overline{P}$, then for every $a\in A$ the limit
$$
\lim_{t\to\infty}e^{-t(w_{0}\lambda+\rho)(X)}m_{v,\eta_{\lambda}}\big(a\exp(tX)\big)
$$
exists and is equal to $J_{w_{0}}(-\lambda)v(a)$.
Moreover, let $\mu(X):=\max\{\alpha(X)\mid\alpha\in \Sigma^{+}\}$. Then for every compact subset $B\subseteq A$  there exists a $C>0$ so that for every $a\in B$ and $t\geq 0$
$$
\Big|m_{v,\eta_{\lambda}}\big(a\exp(tX)\big)- J_{w_{0}}(-\lambda)v\big(a\exp(tX)\big)\Big|
\leq C e^{t\big(\Re w_{0}\lambda+\rho+\mu\big)(X)}q(v)
$$
with
$$
q(v)=\int_{N}\|\log (n)\| \big|v\big(\omega_{0}n\big)\big|\,dn.
$$
\end{theorem}

\begin{proof}
By Lemma \ref{Lemma Integral identity for Phi} the integral over $K$ in  (\ref{pairing}) can be rewritten as
$$
\int_{K/M} \big(f_1(k), f_2(k)\big)_\sigma \,dk
=\int_{N}\big(f_1(n), f_2(n)\big)_\sigma\, dn
$$
for all $\sigma\in\hat{M}$,  $f_{1}\in \cH_{\sigma, \lambda}^\infty$ and $f_{2}\in \cH_{\sigma', -\lambda}^\infty$.
We apply this to the matrix coefficient $m_{v,\eta_{\lambda}}$. We thus obtain for $g\in G$
$$
m_{v,\eta_{\lambda}}(g)
=m_{v,\eta_{\lambda}}(g\omega_{0})
=\int_{N}\eta_{\lambda}(n) v(g\omega_{0}n)\,dn
=\int_{N}\acomp(n)^{-\lambda+\rho} v(g\omega_{0}n)\,dn.
$$
For every $a\in A$ it follows that
\begin{align}
m_{v,\eta_{\lambda}}(a)
\nonumber&=\int_{N}\acomp(n)^{-\lambda+\rho} v(a\omega_{0}n)\,dn\\
\nonumber&=(w_{0}^{-1}\cdot a)^{\lambda+\rho}\int_{N}\acomp(n)^{-\lambda+\rho} v\big(\omega_{0}(w_{0}^{-1}\cdot a)n(w_{0}^{-1}\cdot a)^{-1}\big)\,dn\\
\label{eq Expression matrix coeff}
    &=(w_{0}^{-1}\cdot a)^{\lambda-\rho}\int_{N}\acomp\big((w_{0}^{-1}\cdot a)^{-1}n(w_{0}^{-1}\cdot a)\big)^{-\lambda+\rho} v\big(\omega_{0}n\big)\,dn.
\end{align}
For the last equality we used a change of variables.
For every $n\in N$ the element $\exp\big(-tw_{0} X\big)n\exp\big(tw_{0} X\big)$ converges for $t\to\infty$ to $\1$ and hence
$$
\lim_{t\to\infty}\acomp\Big(\exp\big(-tw_{0} X\big)n\exp\big(tw_{0} X\big)\Big)^{-\lambda+\rho}=1 \qquad(n\in N).
$$
Since $\supp(v)\subset \omega_{0}N\overline{P}$, the function  $N\ni n\mapsto v(\omega_{0}n)$ has compact support.
By the Lebesgue dominated convergence theorem we therefore obtain
$$
\lim_{t\to\infty}e^{-t(w_{0}\lambda+\rho)(X)}m_{v,\eta_{\lambda}}\big(\exp(tX)\big)
=\int_{N} v\big(\omega_{0}n\big)\,dn.
$$
The latter is equal to $J_{w_{0}}(-\lambda)v(\1)$.

Let now $a\in A$. Then
\begin{align*}
&\lim_{t\to\infty}e^{-t(w_{0}\lambda+\rho)(X)}\,m_{v,\eta_{\lambda}}\big(a\exp(tX)\big)
=\lim_{t\to\infty}e^{-t(w_{0}\lambda+\rho)(X)}\,m_{\pi_{-\lambda}(a^{-1})v,\eta_{\lambda}}\big(\exp(tX)\big)\\
&\qquad=J_{w_{0}}(-\lambda)\big(\pi_{-\lambda}(a^{-1})v\big)(\1)
=\pi_{-w_{0}\lambda}(a^{-1})J_{w_{0}}(-\lambda)v(\1)
=J_{w_{0}}(-\lambda)v(a).
\end{align*}

We move on to the error estimate. For $a\in A$ we define
$$
R(a):=\Big|m_{v,\eta_{\lambda}}(a)- J_{w_{0}}(-\lambda)v(a)\Big|.
$$
By (\ref{eq Expression matrix coeff}) we have
\begin{align*}
R(a)
& = \Big|a^{w_{0}\lambda+\rho}\int_{N}\acomp\big((w_{0}^{-1} a)^{-1}n(w_{0}^{-1} a)\big)^{\lambda+\rho} v\big(\omega_{0}n\big)\,dn
    -\int_{N} v\big(a\omega_{0}n\big)\,dn\Big|\\
&=a^{w_{0}\Re\lambda+\rho}\Big|\int_{N}\Big(\acomp\big((w_{0}^{-1} a)^{-1}n(w_{0}^{-1} a)\big)^{\lambda+\rho} -1\Big) v\big(\omega_{0}n\big)\,dn\Big|.
\end{align*}
By the theorem of Taylor there exists a $c>0$ so that
$$
\big|\acomp\big((w_{0}^{-1} a)^{-1}n(w_{0}^{-1} a)\big)^{\lambda+\rho} -1\big|
\leq c\|\Ad\big((w_{0}^{-1} a)^{-1}\big)\log(n)\|
$$
for all $n\in N$ in a sufficiently small neighborhood of $\1$ in $N$.
Note that
$$
\|\Ad\big((w_{0}^{-1} a)^{-1}\big)Y\|
\leq \max \{a^{\alpha}\mid\alpha\in\Sigma^{+}\}\|Y\|
\qquad (Y\in\fn, a\in A).
$$
Therefore,
$$
R(a)
\leq  a^{w_{0}\Re\lambda+\rho}\max \{a^{\alpha}\mid\alpha\in\Sigma^{+}\}\int_{N}\|\log (n)\| \big|v\big(\omega_{0}n\big)\big|\,dn
\qquad(a\in A),
$$
and hence
$$
R\big(a\exp(tX)\big)
\leq C e^{t\big(w_{0}\Re\lambda+\rho+\mu\big)(X)}q(v)
\qquad(a\in B, t\geq 0),
$$
where
$$
C
=c \max_{a\in B}\Big(a^{w_{0}\Re\lambda+\rho}\max \{a^{\alpha}\mid\alpha\in\Sigma^{+}\}\Big).
$$
\end{proof}

\begin{cor}\label{Cor asymptotic behavior of matrix coefficients}
Let $\lambda\in\fa_{\C}^{*}$ and $X\in\fa^{--}$.  Let $w_{0}\in W$ be the longest element and let $\omega_{0}\in N_{K}(\fa)$ be a representative for $w_{0}$. There exists a $v\in \cH_{-\lambda}^{\infty}$ with $\supp(v)\subset \omega_{0}N\overline{P}$ and for every $\epsilon>0$ there exists a $c>0$ so that
$$
\frac{1}{c}e^{t(w_{0}\Re\lambda+\rho)(X)-t\epsilon}
\leq \big|m_{v,\eta_{\lambda}}\big(\exp(tX)\big)\big|
\leq ce^{t(w_{0}\Re\lambda+\rho)(X)+t\epsilon}
\qquad(t\geq0).
$$
\end{cor}

\begin{proof}
The corollary follows directly from the first assertion in  Theorem  \ref{Thm Principal asymptotics} by choosing $v$ so that $J_{w_{0}}(-\lambda)v(\1)\neq0$.
\end{proof}

\subsection{$Z$-Temperedness revisited}\label{Subsection Temperedness revisited}
As an application of the theory of principal asymptotics we provide here the proof for the remaining implication in Theorem \ref{Thm classification of tempered reps}: for every tempered irreducible spherical unitary representation $\pi$ of $G$ there exists a $\lambda\in i\af^*$ so that $[\pi]=[\pi_\lambda]$.

\begin{proof}
The dual $\pi'$ of $\pi$ is also spherical.
By Theorem \ref{Thm Casselman subrep thm} there exists a $\lambda\in \fa_{\C}^{*}$ so that $\pi'$ embeds $G$-equivariantly into $\cH_{\lambda}$.
The space of $K$-fixed vectors in $\cH_{\pi'}$ is  identified with $\C\eta_{\lambda}$. By taking duals we find that $\pi$ is a quotient of $\cH_{-\lambda}$. If $u\in \cH_{\pi}$ and $v\in\cH_{-\lambda}$ is in the preimage of $u$, then $m_{u,\eta_{\lambda}}=m_{v,\eta_{\lambda}}$. Let $X\in \fa^{--}$ and let $w_{0}$ be the longest Weyl group element. By Corollary \ref{Cor asymptotic behavior of matrix coefficients} there exists a $v\in\cH_{-\lambda}^{\infty}$ and for every $\epsilon>0$ there exists a $c>0$ so that
$$
|m_{v,\eta_{\lambda}}\big(\exp(tX)\big)|\geq c e^{t(w_{0}\Re\lambda+\rho)(X)-t\epsilon} \qquad (t\geq0).
$$
Since $\pi$ is tempered, it follows from (\ref{tempered}) that there exists an $r\in \R$ so that
$$
\sup_{a\in A}|m_{v,\eta_{\lambda}}(a)|a^{-\rho}(1+\|\log(a)\|)^{r}<\infty.
$$
This implies that $w_{0}\Re\lambda(X)\leq 0$. Since this holds for all $X\in \fa^{--}$, it follows that
\begin{equation}\label{eq first inequality}
\Re\lambda\in \sum_{\alpha\in\Sigma^{+}}\R_{\leq0}\alpha.
\end{equation}

We will now apply a standard intertwining operator and repeat the above argument.
We claim that there exists a $w\in W$ so that $\langle\Re\lambda,\alpha\rangle> 0$ for $\alpha\in -\Sigma^{+}$, if and only if $\alpha\in -\Sigma^{+}\cap w^{-1}\Sigma^{+}$.  To prove this we define
$$
S
:=\{\alpha\in \Sigma\mid \langle\Re\lambda,\alpha\rangle>0\}\cup\{\alpha\in \Sigma^{+}\mid \langle\Re\lambda,\alpha\rangle=0\}.
$$
It is easy to see that $S$ is a positive system of $\Sigma$. Therefore, there exists a $w\in W$ so that $\Sigma^{+}=wS$. This element $w$ satisfies the conditions in the claim.

The vector $J_{w}(\lambda)v$ is for every $v\in \cH_{\lambda}^{\infty}$ defined by absolutely convergent integrals by (\ref{eq lambda condition}) and it follows from the formula of Gindikin and Karpelevich that $\cfunc_{w}$ does not have any zero's. Therefore, the vector $J_{w}(\lambda)\eta_{\lambda}=\cfunc_{w}(\lambda)\eta_{w\lambda}$ is a non-zero multiple of $\eta_{w\lambda}\in\cH_{w\lambda}$.

Let $X\in \fa^{--}$.
By Corollary \ref{Cor asymptotic behavior of matrix coefficients} with $w\lambda$ in place of $\lambda$ there exists a $v\in\cH_{-w\lambda}^{\infty}$ and for every $\epsilon>0$ there exists a $c>0$ so that
\begin{equation}\label{eq Lower bound matrix coeff}
|m_{v,\eta_{w\lambda}}\big(\exp(tX)\big)|\geq c e^{t(w_{0}w\Re \lambda+\rho)(X)-t\epsilon} \qquad (t\geq0).
\end{equation}
In view of Lemma \ref{Lemma dual of intertwiner} we have
$$
m_{v,\eta_{w\lambda}}
=\frac{1}{\cfunc_{w}(\lambda)}m_{J_{w^{-1}}(-w\lambda)v, \eta_{\lambda}}.
$$
It thus follows from (\ref{eq Lower bound matrix coeff}) that for every $\epsilon>0$ there exists a $v\in \cH_{-w\lambda}^{\infty}$ and a $C>0$ so that
$$
|m_{J_{w^{-1}}(-w\lambda)v, \eta_{\lambda}}\big(\exp(tX)\big)|
\geq Ce^{t(w_{0}w\Re \lambda+\rho)(X)-t\epsilon}
\qquad(t\geq 0).
$$
As before, since $\pi$ is tempered, we have
$$
\sup_{a\in A}|m_{J_{w^{-1}}(-w\lambda)v, \eta_{\lambda}}(a)|a^{-\rho}(1+\|\log(a)\|)^{r}<\infty.
$$
This implies that $w_{0}w\Re \lambda(X)\leq 0$ for all $X\in \fa^{--}$, and hence
$$
\Re\lambda
\in w^{-1} \Big(\sum_{\alpha\in\Sigma^{+}}\R_{\leq0}\alpha\Big).
$$
In view of (\ref{eq first inequality}) we thus have
$$
\Re\lambda
\in \Big(\sum_{\alpha\in\Sigma^{+}}\R_{\leq0}\alpha\Big)\cap w^{-1} \Big(\sum_{\alpha\in\Sigma^{+}}\R_{\leq0}\alpha\Big)
=\sum_{\alpha\in \Sigma^{+}\cap w^{-1}\Sigma^{+}}\R_{\leq 0}\alpha.
$$
The condition on $w$ guarantees that $\langle\Re\lambda,\alpha\rangle\geq 0$ for all $\alpha\in \Sigma^{+}\cap w^{-1}\Sigma^{+}$. For $\alpha\in \Sigma^{+}\cap w^{-1}\Sigma^{+}$ let $c_{\alpha}\leq 0$ be so that $\Re\lambda=\sum_{\alpha\in \Sigma^{+}\cap w^{-1}\Sigma^{+}}c_{\alpha}\alpha$. Then
$$
\|\Re\lambda\|^{2}
=\sum_{\alpha\in \Sigma^{+}\cap w^{-1}\Sigma^{+}}c_{\alpha}\langle\Re\lambda,\alpha\rangle
\leq 0.
$$
It thus follows that $\Re\lambda=0$.
This finishes the proof of Theorem \ref{Thm classification of tempered reps}.
\end{proof}

\section{The Plancherel formula for $L^2(Z_\empt)$}\label{Section Planch for Z_empt}

In this section we introduce a homogeneous space $Z_{\empt}$ of $G$, which is asymptotically similar to $Z$.
The explicit Plancherel decomposition of $L^{2}(Z_{\empt})$ is fairly easy to derive.
In Section \ref{Section Proof of Plancherel formula} we will see that the abstract Plancherel formula for $L^2(Z)$ induces a decomposition of $L^2(Z_\empt)$. By using the result of the current section on the Plancherel formula for $L^2(Z_\empt)$ we will then be able to derive an explicit Plancherel decomposition of $L^2(Z)$.

\subsection{Boundary degenerations}
Let $d:=\dim \kf$ and write $\Gr_d(\gf)$ for the Grassmannian of $d$-dimensional subspaces of $\gf$. Define $\kf_\empt:= \mf + \overline{\nf}$ and note that
$\dim \kf_\empt =d$.
The subalgebra $\fk_{\empt}$ can be obtained from $\fk$ by a limiting process in $\Gr_d(\gf)$ as
$$
\lim_{t\to \infty} e^{t\ad X} \kf =\kf_\empt
$$
where $X$ is any element in $\af^{--}$. This is a consequence of
$$
\kf
= \mf \oplus \bigoplus_{\alpha \in\Sigma^+} (\gf^\alpha + \gf^{-\alpha})^\theta
= \mf \oplus \bigoplus_{\alpha \in\Sigma^+} (1+\theta)\gf^{-\alpha}.
$$
We call $Z_\empt = G/M\overline{N}$ the (minimal) boundary degeneration of $Z$.

The important observation is that both $Z$ and the open $P$-orbit in $Z_{\empt}$ are isomorphic to $P/M$ as $P$-spaces. The advantage of $Z_{\empt}$ over $Z$ is that it features more symmetries: As $A$ normalizes $M\overline{N}$ we obtain for each $a\in A$ a $G$-equivariant automorphism $ gM\overline{N} \mapsto gaM\overline{N}$ of $Z_\empt$. This fact greatly simplifies the derivation of the Plancherel decomposition of $L^{2}(Z_{\empt})$.

The space $\hat Z_\empt=G/\overline{P}$ is a compact homogeneous space for $G$. Let $\overline{Z}$ be a compact smooth $G$-manifold that contains $Z$ as an open $G$-orbit and $\hat{Z}_{\empt}$ as the unique closed orbit. See for example \cite[Sect.14]{KK} for a construction of such a manifold under the assumption that $G$ is algebraic and of adjoint type. One can regard the closed $G$-orbit in $\overline{Z}$ as the orbit at ''infinity''. The $G$-space $Z_\empt$ is naturally identified with the part of the normal bundle of $\hat Z_\empt$ which points to $Z$. In this sense the space $Z_\empt$ approximates $Z$ geometrically at infinity.

We illustrate this with an example for $G=\SO_e(1,2)$ the connected component of $\SO(1,2)$. Note that $G\simeq \PSl(2,\R)$.
Then
$$\overline{Z} =\{ [x]=[x_1, x_2, x_3, x_4]\in \Pb^3(\R)\mid x_1^2 - x_2^2 -x_3^2 =x_4^2\}$$
is a $G$-space induced from the linear action of $G$  on $\R^3$ trivially extended to
$\R^4 = \R^3 \times \R$.  Observe that
$$
\overline{Z}=\{ [1, x_2, x_3, x_4]\in \Pb^3(\R)\mid  x_2^2 +x_3^2 +x_4^2=1\}
$$
is diffeomorphic to the sphere $S^2$. The orbit decomposition of $\overline{Z}$ is
$$
\overline{Z}
=Z^{+}\sqcup\hat{Z}_{\empt}\sqcup Z^{-},
$$
where
$$
Z^{\pm}=\{ [x]=[1, x_2, x_3, x_{4}]\in \Pb^3(\R)\mid  x_2^2 +x_3^2+x_{4}^{2}=1, \pm x_{4}>0\}=G\cdot [e_{1}\pm e_{4}]
$$
naturally correspond to the upper and lower hemisphere of $S^{2}$, and
$$
\hat{Z}_{\empt}=\{ [x]=[x_1, x_2, x_3, 0]\in \Pb^3(\R)\mid x_1^2 - x_2^2 -x_3^2=0\}=G\cdot [e_{1}+e_{3}]
$$
to the equator.
The stabilizers of $[e_{1}\pm e_{4}]$ are both equal to $K$ and the stabilizer of  $[e_{1}+e_{3}]$ is $\overline{P}$. Therefore, $Z^{\pm}\simeq Z$ and $\hat{Z}_{\empt}\simeq G/\overline{P}$.

The tangent space of $\Pb^{3}(\R)$ at a point $[x]$ may be identified with the space $\Hom_{\R}([x],\R^{4}/[x])$. The action of $G$ on $T\Pb^{3}(\R)$ is given by
$$
g\cdot ([x],\xi)=([g\cdot x], g\circ\xi\circ g^{-1})\qquad \big(g\in G, ([x],\xi)\in T\Pb^{3}(\R)\big).
$$
If $[x]\in\hat{Z}_{\empt}$, then the tangent space $T_{[x]}\hat{Z}_{\empt}$ corresponds to the subspace $\Hom_{\R}(\R,(\R^{3}\times\{0\})/[x])$ under this identification.
The normal bundle ${\mathsf N}:= N_{\hat Z_\empt} \overline{Z} $ to the closed orbit is therefore given by
$$
{\mathsf N} =
\Big\{ ([x],\xi)\in T\Pb^{3}(\R)\mid x\in\hat{Z}_{\empt}, \xi \in \Hom_{\R}\big([x],\mathrm{span}_{\R}(x,e_{4})/[x]\big)\Big\}.
$$
Note that ${\mathsf N}$ is stable under the action of $G$ on $T\Pb^{3}(\R)$ since $G$ acts trivially on the last coordinate in $\R^{4}$.
For every $[x]\in\hat{Z}_{\empt}$ and $v\in\mathrm{span}_{\R}(x,e_{4})/[x]$ there exists a unique $v_{4}\in\R$ so that $v=v_{4}e_{4}+[x]$. Given $\xi \in \Hom_{\R}\big(\R,\mathrm{span}_{\R}(x,e_{4})/[x]\big)$, let $\xi_{4}\in\R$ be so that $\xi(1)=\xi_{4}e_{4}+[x]$. Let $\xi^{\circ}\in \Hom_{\R}\big([e_{1}+e_{3}],\mathrm{span}_{\R}(e_{1}+e_{3},e_{4})/[e_{1}+e_{3}]\big)$ be the unique element so that $(\xi^{\circ})_{4}=1$.
The orbit decomposition of ${\mathsf N}$ is
$$
{\mathsf N}={\mathsf N}^{+}\sqcup {\mathsf N} ^{0}\sqcup{\mathsf N}^{-},
$$
where
\begin{align*}
&{\mathsf N}^{\pm}
=\{([x],\xi)\in{\mathsf N} \mid \pm\xi_{4}>0\}= G\cdot\big([e_{1}+e_{3}],\pm \xi^{\circ}\big)\simeq G/\overline{N},\\
&{\mathsf N}^{0}
=\{([x],0)\in T\Pb^{3}(\R)\mid [x]\in\hat{Z}_{\empt}\}= G\cdot \big([e_{1}+e_{3}],0\big)\simeq G/\overline{P}.
\end{align*}
Note that $M=\{e\}$ for our choice of $G$, so that $G/\overline{N}=G/M\overline{N}$.
The spaces ${\mathsf N}^{\pm}$ are the parts of the normal bundle of $\hat{Z}_{\empt}$ which point to the open orbits $Z^\pm$.

\subsection{A first decomposition}
First we introduce polar coordinates on $Z_\empt =G/M\overline{N}$. Let $z_\empt=M\overline{N}$ be the standard base point of $Z_\empt=G/M\overline{N}$. From the Iwasawa decomposition $G=KA\overline{N}$ we obtain that the polar map
$$K/M\times A \to Z_\empt, \quad (kM, a)\mapsto ka\cdot z_\empt$$
is a diffeomorphism.
In view of (\ref{eq Integral formula Iwasawa decomp}) the invariant measure $dz$ on $Z_\empt$ satisfies
\begin{equation}\label{int polar}
 \int_{Z_\empt} f(z) \,dz
=\int_{K/M}\int_A f(ka\cdot z_\varnothing)\  a^{-2\rho} \,da \,dk\qquad \big(f \in C_c(Z_\empt)\big).
\end{equation}
The normalized right regular action
\begin{equation}\label{eq normalized right action on Z_empt}
(R(a)f) (gM\overline{N}) =a^{-\rho}f(gaM\overline{N})\qquad \big(a\in A, f\in L^2(Z_\empt)\big)
\end{equation}
defines a unitary representation of the abelian group $A$
$$ A \times L^2(Z_\empt)\to L^2(Z_\empt), \quad (a, f)\mapsto R(a)f.$$
We identify $\hat A$ with $i\af^*$ as usual and decompose $L^2(Z_\empt)$ with respect to $A$. For $f\in C_c^\infty(Z_\empt)$ and $\lambda\in i\af^*$ we set
$$\hat f(\lambda)(z)= \int_A \big(R(a)f\big)(z) a^{\lambda}\,da \qquad (z\in Z_\empt)$$
and note that $\hat f(\lambda)\in \cH_\lambda^\infty$, see \eqref{def princ}. Moreover we obtain from \eqref{int polar} and the Plancherel formula for $A$ that $f\mapsto \hat{f}$ extends continuously to $L^{2}(Z_{\empt})$ and
$$
 \|f\|_{L^2(Z_\empt)}^2= \int_{i\af^*} \|\hat f(\lambda)\|_{L^2(K/M)}^2 \,d\lambda\qquad \big(f\in L^{2}(Z_{\empt})\big),
$$
where $d\lambda$ is the Lebesgue measure on $i\af^*$. Recall that
 $\|\hat f(\lambda)\|_{L^2(K/M)}= \|\hat f(\lambda)\|_{\cH_\lambda}$ for all $\lambda\in i\af^*$.
To summarize, the left regular representation $L$ of $G$ on $L^{2}(Z_{\empt})$ decomposes into a direct integral of representations $\cH_{\lambda}$. In fact we have a $G$-equivariant unitary isomorphism
 \begin{equation} \label{planch empty 1}
 \big(L, L^2(Z_\empt)\big)\to \left(\int_{i\af^*}^\oplus \pi_\lambda \,d\lambda, \int_{i\af^*}^\oplus \cH_\lambda \,d\lambda\right),
  \quad f \mapsto \big (\lambda\mapsto \hat f(\lambda)\big).
 \end{equation}

Although this a quite reasonable way to write the Plancherel decomposition of $L^{2}(Z_{\empt})$, it is unsatisfactory for the following reason. For any given $\lambda\in i\fa^{*}$ the representations $\pi_\lambda$ and $\pi_{w\lambda}$ are equivalent for every $w\in W$. Hence we wish to rewrite the integral over $i\af^*$  to an integral over the fundamental domain $i\af_+^*\subset i\af^*$ for the $W$-action so that we avoid interference between the irreducible representations in the direct integral. Doing so will give us multiplicity spaces of generic dimension $|W|$.

\subsection{Conical generalized vectors}

We start by constructing the multiplicity spaces which should be given by spaces of $M\overline{N}$-fixed vectors of some sort.
However, if $\lambda\in i\fa^{*}$ then there are no non-trivial $M\overline{N}$-fixed vectors in $\cH_{\lambda}$. In fact, if $f\in \cH_{\lambda}^{M\overline{N}}\setminus\{0\}$ and $\omega_{0}\in N_{K}(\fa)$ is a representative of the longest Weyl group element in $W$,
then in view of Lemma \ref{Lemma Integral identity for Phi} we have
\begin{align*}
\|f\|^{2}
&=\|\pi_{\lambda}(\omega_{0}^{-1})f\|^{2}
=\int_{N} |f(\omega_{0}n)|^{2} \,dn
=\int_{\overline{N}}|f(\overline{n}\omega_{0})|^{2}\,d\overline{n}\\
&=\int_{\overline{N}} |\pi_{\lambda}(\overline{n}^{-1})f(\omega_{0})|^{2}\,d\overline{n}
=|f(\omega_{0})|^{2}\int_{\overline{N}} \,d\overline{n}.
\end{align*}
The latter integral is divergent, which leads to a contradiction.
This phenomenon is a special case of the Howe-Moore vanishing theorem. See \cite[Section V.2]{HoweTan}.
The remedy for this is to use the space $\cH_{\lambda}^{-\infty}$ of generalized vectors.

For $\lambda, \mu \in\af_{\C}^{*}$ we define the space
$$
(\cH_{\lambda}^{-\infty})^{\overline{P}, \mu}
:=\Big\{ \xi \in \cH_{\lambda}^{-\infty}\mid \pi_{\lambda}(\overline{p})\xi
= a^{\rho +\mu}\xi \text{ for all }\overline{p} = m a\overline{n} \in \overline{P}
= M A\overline{N}\Big\}.
$$
The elements of  $(\cH_{\lambda}^{-\infty})^{\overline{P}, \mu}$ are called conical distributions (see \cite[Ch. II, § 5]{Hel}).
Let
\begin{equation}\label{eq Def xi^1}
\xi_{\lambda}^{\1}: \cH_{-\lambda}^\infty\to \C, \quad f\mapsto f(\1).
\end{equation}
Note that $\xi_{\lambda}^{\1}\in (\cH_{\lambda}^{-\infty})^{\overline{P}, \lambda}$.

Let $\lambda\in i\fa^{*}$.
By $G$-equivariance we have
$$
I_w(\lambda)\big((\cH_{\lambda}^{-\infty})^{\overline{P}, \mu}\big)= (\cH_{w\lambda}^{-\infty})^{\overline{P}, \mu}
\qquad\big(\mu\in \af_{\C}^*,w\in W\big).
$$
For $w\in W$ we define
\begin{equation}\label{eq Def xi^w}
\xi_\lambda^{w}
:=  I_{w^{-1}}(w\lambda)\xi_{w\lambda}^{\1}\in (\cH_\lambda^{-\infty})^{\overline{P}, w\lambda}.
\end{equation}

If $\lambda$ is regular, i.e. $\langle\lambda,\alpha\rangle\neq 0$ for all $\alpha\in\Sigma$, then we obtain from \cite[Ch.II, Th. 5.15]{Hel} that $(\cH_{\lambda}^{-\infty})^{M\overline{N}}$ is finite dimensional and
\begin{equation} \label{conic}
(\cH_{\lambda}^{-\infty})^{M\overline{N}}
= \bigoplus_{w\in W} (\cH_\lambda^{-\infty})^{\overline{P}, w\lambda}=\bigoplus_{w\in W} \C \xi_\lambda^{w}.
\end{equation}
(The condition in the theorem that $\lambda$ is simple is automatically satisfied for $\lambda\in i\fa^{*}$. In fact, the Poisson transform of the constant function $1$  on $K/M$ equals the spherical function $\phi_{\lambda}$. Therefore the Poisson transform is non-zero. Since $\pi_{\lambda}$ is irreducible the Poisson transform is injective.)

For $\lambda \in i\fa^{*}$ a regular element, we put a Hilbert space structure on $(\cH_{\lambda}^{-\infty})^{M\overline{N}}$ by requiring that $(\xi_\lambda^{w})_{w\in W}$ is an orthonormal basis.
We equip $\Hom \big((\cH_{\lambda}^{-\infty})^{M\overline{N}}, \cH_\lambda\big)$ with the Frobenius norm
$$
\Hom \big((\cH_{\lambda}^{-\infty})^{M\overline{N}}, \cH_\lambda\big)\to \R_{\geq0},\quad T\mapsto \tr( T^{\dagger}\circ T).
$$
We define the multiplicity space by
$$
\cM_\lambda
:=\big((\cH_\lambda^{-\infty})^{M\overline{N}}\big)'
$$
and equip it with the Hilbert structure induced by $(\cH_\lambda^{-\infty})^{M\overline{N}}$.
Then
$$
\Hom \big((\cH_{\lambda}^{-\infty})^{M\overline{N}}, \cH_\lambda\big)
\simeq \cH_\lambda\otimes \cM_\lambda
$$
as Hilbert spaces.

\subsection{The Plancherel theorem}

We now arrive at the following reformulation of the Plancherel theorem \eqref{planch empty 1}. Let
$$
\af_{+}^{*}
:=\{\lambda\in \fa^{*}\mid\langle\lambda,\alpha\rangle\geq0 \text{ for all }\alpha\in\Sigma^{+}\}.
$$

\begin{theorem}\label{Thm Planch thm for Z_empt}
For $f\in C_{c}(Z_{\empt})$ and $\lambda\in i\fa^{*}_{+}$ let 
$$
\cF_{\empt} f(\lambda)\in \cH_\lambda \otimes\cM_{\lambda}\simeq  \Hom \big((\cH_{\lambda}^{-\infty})^{M\overline{N}}, \cH_\lambda\big)
$$
be given by
$$
\cF_{\empt} f(\lambda)\xi= \pi_{\lambda}(f)\xi\qquad \big( \xi\in (\cH_{\lambda}^{-\infty})^{M\overline{N}}\big).
$$
Then $\cF_{\empt}$ extends uniquely  to a $G$-equivariant unitary isomorphism
\begin{align} \label{planch empty 2}
\cF_{\empt}: \big(L, L^2(Z_\empt)\big)&\to
    \left(\int_{i\af_+^*}^\oplus \pi_\lambda\otimes {\rm id}_{\cM_{\lambda}}\,d\lambda, \int_{i\af^*_{+}}^\oplus \cH_\lambda \otimes\cM_{\lambda}\,d\lambda\right).
\end{align}
\end{theorem}

\begin{proof}
We claim that
\begin{equation}\label{Fourier Fourier}
\hat f(\lambda)= \pi_{\lambda}(f)\xi_{\lambda}^{\1}
=\int_{G/M\overline{N}}f(g\cdot z_{\empt})\pi_{\lambda}(g)\xi_{\lambda}^{\1}\,d(gM\overline{N})
    \qquad \big(f \in C_c^\infty(Z_\empt)\big).
\end{equation}
This is verified by testing against $v \in \cH_{-\lambda}^\infty$. We have
\begin{align*}
\big( \pi_{\lambda}(f)\xi_{\lambda}^{\1} , v\big)
&= \int_{G/M\overline{N}} f(g\cdot z_\empt) \big(\pi_{\lambda}(g)\xi_{\lambda}^{\1}, v\big) \,d (gM\overline{N}) \\
&=\int_{G/M\overline{N}}  f(g\cdot z_\empt) \xi_{\lambda}^{\1}\big(\pi_{-\lambda}(g^{-1})v\big)\,d(gM\overline{N})\\
&=\int_{G/M\overline{N}}  f(g\cdot z_\empt)v(g)\,d(gM\overline{N})\\
&\mathrel{\underset{(\ref{int polar})}{=}}\int_{K/M}\int_A f(ka\cdot z_\empt)v(ka)\ a^{-2\rho} \,da \,d(kM)\\
&=\int_{K/M}\int_A f(ka\cdot z_\empt)v(k) a^{\rho+\lambda} a^{-2\rho} \,da \,d(kM)\\
&\mathrel{\underset{(\ref{eq normalized right action on Z_empt})}{=}}\int_{K/M}\int_A (R(a)f)(k\cdot z_\empt)v(k) a^{\lambda} \,da \,d(kM)\\
&=\int_{K/M}\hat f (\lambda)(k) v(k) \,d(kM)
=\big( \hat f(\lambda), v\big).
\end{align*}
This proves the claim.

Since $I_{w}(\lambda)$ is unitary for every $w\in W$, we have
$$
\|\pi_{\lambda}(f)\xi_{\lambda}^{w}\|
=\|I_{w}(\lambda) \big(\pi_{\lambda}(f)\xi_{\lambda}^{w}\big)\|
=\|\pi_{w\lambda}(f) I_{w}(\lambda)(\xi_{\lambda}^{w})\|
=\|\pi_{w\lambda}(f)\xi_{w\lambda}^{\1}\|,
$$
and hence it follows from \eqref{Fourier Fourier} that
$$
\|\hat{f}(w\lambda)\|= \|\pi_{w\lambda}(f)\xi_{w\lambda}^{\1}\|= \|\pi_{\lambda}(f)\xi_\lambda^{w}\|
\qquad(w\in W).
$$
Therefore,
\begin{align*}
\int_{i\fa_{+}^{*}}\|\cF_{\empt} f(\lambda)\|^{2}\,d\lambda
&=\int_{i\fa_{+}^{*}}\sum_{w\in W}\|\pi_{\lambda}(f)\xi_\lambda^{w}\|^{2}\,d\lambda\\
&=\sum_{w\in W}\int_{i\fa_{+}^{*}}\|\hat{f}(w\lambda)\|^{2}\,d\lambda
=\int_{i\fa^{*}}\|\hat{f}(\lambda)\|^{2}\,d\lambda,
\end{align*}
and hence it follows from the decomposition  \eqref{planch empty 1} that (\ref{planch empty 2}) is a unitary isomorphism.
\end{proof}

\section{Proof of the Plancherel formula}\label{Section Proof of Plancherel formula}
The goal is now to give a sketch of the proof of the following Plancherel theorem of Harish-Chandra for the left regular representation $L$ of $G$ on $L^{2}(G/K)$.
Recall that $w_{0}$ is the longest Weyl group element in $W$. We write $\cfunc$ for the $\cfunc$-function $\cfunc_{w_{0}}$.

\begin{theorem} \label{thm planch}
For all $\lambda\in i\af^*$ the Maass-Selberg relations
$$
|\cfunc(\lambda)|
=|\cfunc (w\lambda)| \qquad (w\in W)
$$
hold.
Furthermore, for $f\in C_{c}(Z)$ and $\lambda\in i\fa^{*}_{+}$ let 
$$
\cF f(\lambda):= \pi_{\lambda}(f)\eta_{\lambda}\in \cH_\lambda.
$$
Then $\cF$ extends uniquely to a $G$-equivariant unitary isomorphism
$$
 \cF:\big(L, L^2(Z)\big)\to \left(\int_{i\af^*_{+}}^\oplus \pi_\lambda \,\frac{ d\lambda}{|\cfunc(\lambda)|^{2}},
     \int_{i\af^*_{+}}^\oplus \cH_\lambda\,\frac{ d\lambda}{|\cfunc(\lambda)|^{2}}\right).
 $$
\end{theorem}

\subsection{Constant term approximation}
In this section we describe Harish-Chandra's theory of the constant term for matrix coefficients $m_{v,\eta_{\lambda}}$ with $\lambda\in i\fa^{*}$ and  $v\in \cH_{-\lambda}$. The constant term map is a more sophisticated version of the principal asymptotics from Theorem \ref{Thm Principal asymptotics} for unitary principal series representations. What we cite below is a special case of a general theorem in \cite{DKS}.

In view of the Theorem \ref{Thm classification of tempered reps} there exists a number $r \in \R$ such that
$$
q_{\lambda}(v):=\sup_{z\in Z} |m_{v, \eta_{\lambda}}(z)| \sqrt{{\bf v}(z)} {\bf w}(z)^r <\infty \qquad
(\lambda\in i\fa^{*},v \in \cH_{-\lambda}^\infty).
$$
For $k \in \N$ we denote by $q_{\lambda,k}$ a $k$-th Sobolev norm of $q_{\lambda}$, i.e, the norm
$$
q_{\lambda,k}(v):=\max_{u\in \mathcal{B}_{k}}q_{\lambda}(u\cdot v)\qquad (v\in \cH_{-\lambda}^{\infty}),
$$
where $\mathcal{B}_{k}$ is a basis of $\{u\in \cU(\fg)\mid \deg(u)\leq k\}$.
We recall that $(\cH_{\lambda}^{-\infty})^K= \cH_{\lambda}^K$.

\begin{theorem} \label{thm const}
There exists a family of unique linear maps
$$  (\cH_{\lambda}^{-\infty})^K \to (\cH_{\lambda}^{-\infty})^{M\overline{N}}, \quad \eta\mapsto \eta^\empt
\qquad(\lambda\in i\af^*)
$$
such that the following holds.
There exists a $k\in \N$ and for every compact subset $S\subset \af^{--}$ there exist $\e, C>0$ so that the remainder estimate
$$
\Big|m_{v,\eta}\big(\exp(tX)\big)- m_{v,\eta^\empt}\big(\exp(tX)\big)\Big|
\leq C e^{t(1+\e)\rho(X)}\|\eta\| q_{\lambda,k}(v)
$$
is valid for all $\lambda\in i\fa^{*}$, $\eta\in (\cH_{\lambda}^{-\infty})^{K}$, $v\in\cH_{-\lambda}^{\infty}$, $X\in S$ and $t\geq 0$.
\end{theorem}

The matrix coefficient $m_{v,\eta^\empt}$ is called the constant term of $m_{v,\eta}$. Likewise, $\eta^{\empt}$ is called the constant term of $\eta$.

\begin{cor}\label{Cor Formula constant term}
For every regular $\lambda\in i\fa^{*}$ the constant term of $\eta_{\lambda}$ is given by
$$
\eta_{\lambda}^{\empt}
=\sum_{w\in W}\cfunc(w\lambda)\xi_{\lambda}^{w},
$$
and hence the constant term of $m_{v,\eta_{\lambda}}$ is given by
$$
m_{v,\eta^{\empt}_{\lambda}}(a)
=\sum_{w\in W} \cfunc(w\lambda) \xi_\lambda^{w}(v) a^{\rho +w\lambda}
\qquad(a\in A).
$$
\end{cor}

\begin{proof}
Let $\lambda\in i\af^*$ be regular. In view of \eqref{conic} there exist unique coefficients $\gamma_{w}(\lambda)\in \C$ for $w\in W$, so that
$$
\eta_\lambda^\empt = \sum_{w\in W} \gamma_{w}(\lambda) \xi_\lambda^{w}.
$$
It now remains to show that $\gamma_{w}(\lambda)=\cfunc(w\lambda)$.

Let $v\in \cH_{-\lambda}^{\infty}$ satisfy $\supp(v)\subset \omega_{0}N\overline{P}$. It follows from Theorem \ref{Thm Principal asymptotics} on principal asymptotics that
$$
\eta_{\lambda}^{\empt}(v)
=\big(J_{w_{0}}(-\lambda)v\big)(\1).
$$
In view of (\ref{eq Def xi^1}) we thus have 
$$
\eta_{\lambda}^{\empt}(v)
=\xi_{w_{0} \lambda}^{\1}\big(J_{w_{0}}(-\lambda)v\big).
$$
We now use (\ref{eq Extension of intertwiner to generalized vectors}), (\ref{eq Def I}) and (\ref{eq Def xi^w}) to rewrite the right-hand side and obtain
$$
\eta_{\lambda}^{\empt}(v)
=\big(J_{w_{0}}(w_{0}\lambda)\xi_{w_{0}\lambda}^{\1}\big)(v)\\
=\cfunc(w_{0} \lambda)\big(I_{w_{0}}(w_{0}\lambda)\xi_{w_{0}\lambda}^{\1}\big)(v)
=\cfunc(w_{0} \lambda)\xi_{\lambda}^{w_{0}}(v).
$$
Now for every $a\in A$ we have
\begin{align*}
\big(\pi_{\lambda}(a)\eta_{\lambda}^{\empt}\big)(v)
&=\eta_{\lambda}^{\empt}\big(\pi_{-\lambda}(a^{-1})v\big)
=\cfunc(w_{0} \lambda)\xi_{\lambda}^{w_{0}}\big(\pi_{-\lambda}(a^{-1})v\big)\\
&=\cfunc(w_{0} \lambda)\big(\pi_{\lambda}(a)\xi_{\lambda}^{w_{0}}\big)(v)
=a^{\rho+w_{0}\lambda}\cfunc(w_{0} \lambda)\xi_{\lambda}^{w_{0}}(v).
\end{align*}
Since $\lambda$ is regular, the group $A$ acts on the $\xi_{\lambda}^{w}$ with distinct eigencharacters. Therefore,
$$
\gamma_{w}(\lambda)\xi_{\lambda}^{w}(v)
=\left\{
   \begin{array}{ll}
     0 & (w\in W\setminus\{w_{0}\}), \\
     \cfunc(w_{0} \lambda)\xi_{\lambda}^{w_{0}}(v) & (w=w_{0}).
   \end{array}
 \right.
$$
By choosing $v$ so that $\xi_{\lambda}^{w_{0}}(v)\neq 0$ we find that $\gamma_{w_{0}}(\lambda)=\cfunc(w_{0}\lambda)$.

Since $I_{w}(\lambda)$ is intertwining and continuous, the uniqueness of the constant term map stated in Theorem \ref{thm const} implies that
$$
\eta_{w\lambda}^{\empt}
=\big(I_{w}(\lambda)\eta_{\lambda}\big)^{\empt}
=I_{w}(\lambda)\eta_{\lambda}^{\empt}
\qquad(\lambda\in i\fa^{*}, w\in W).
$$
Moreover, for all $w, w'\in W$ and $\lambda\in i\fa^{*}$ we have
\begin{align*}
I_{w}(\lambda)\xi_\lambda^{w'}
&=  I_{w}(\lambda)I_{w'^{-1}}(w'\lambda)\xi_{w'\lambda}^{\1}
=  I_{ww'^{-1}}(w'\lambda)\xi_{w'\lambda}^{\1}\\
&=  I_{ww'^{-1}}\big((w'w^{-1})w\lambda\big)\xi_{(w'w^{-1})w\lambda}^{\1}
=\xi_{w\lambda}^{w'w^{-1}},
\end{align*}
and hence
\begin{align*}
 \sum_{w'\in W} \gamma_{w'}(\lambda) \xi_{\lambda}^{w'}
&=\eta_{\lambda}^{\empt}
=I_{w}(w^{-1}\lambda)\eta_{w^{-1}\lambda}^{\empt}
=\sum_{w''\in W}\gamma_{w''}(w^{-1}\lambda)I_{w}(w^{-1}\lambda) \xi_{w^{-1}\lambda}^{w''}\\
&=\sum_{w''\in W}\gamma_{w''}(w^{-1}\lambda)\xi_{\lambda}^{w''w^{-1}}.
\end{align*}
From the linear independence of the $\xi_{\lambda}^{w}$ it follows that
$$
\gamma_{w}(\lambda)=\gamma_{w_{0}}(w_{0}^{-1}w\lambda)=\cfunc(w\lambda)
\qquad (w\in W).
$$
\end{proof}

\subsection{Matching of $L^2(Z)$ and $L^2(Z_\empt)$}
Let $z_{0}:=K\in G/K=Z$ and $z_{\empt}:=M\overline{N}\in G/M\overline{N}=Z_{\empt}$.
Note that both $P/M\simeq P\cdot z_{0}=Z$ and
$P/M\simeq P\cdot z_\empt$ as $P$-spaces and that $P\cdot z_\empt$ is open and dense in $Z_{\empt}$.
Because of (\ref{eq Integral formula Iwasawa decomp}) and (\ref{eq Integral formula Bruhat decomp}), the $P$-equivariant matching map
$$
\Phi: C_c(P\cdot z_\empt)\to C_c(Z),
$$
defined by
$$
\Phi(f)(p\cdot z_0):=  f(p\cdot z_\empt) \qquad \big(f\in C_{c}(P\cdot z_{\empt}), p\in P),
$$
extends to a $P$-equivariant unitary equivalence
$$L^2(Z_\empt)\to L^2(Z).$$
We recall the normalized right unitary action $R$ of $A$ on $L^2(Z_\empt)$ from (\ref{eq normalized right action on Z_empt}).
Now for $X\in \af^{--}$, $f\in C_c(P\cdot z_\empt)$ and $t\geq 0$
we set
$$
f_{t,X} := R\big(\exp(-tX)\big)f\in C_c(P\cdot z_\empt).
$$
We note that
$$
\langle f,g\rangle_{L^2(Z_\empt)}= \langle f_{t,X},g_{t,X}\rangle_{L^2(Z_\empt)}=\langle \Phi(f_{t,X}),\Phi(g_{t,X})\rangle_{L^2(Z)}
\qquad \big(f,g\in C_{c}(P\cdot z_{\empt})\big).
$$
In view of Corollary \ref{cor support} the support of the Plancherel measure $\mu$ of $L^{2}(Z)$ is contained in $i\af^*/W$.
By the abstract Plancherel theorem \eqref{eq abstract Plancherel thm} for $L^2(Z)$ we thus have
$$
\langle f,g\rangle_{L^2(Z_\empt)}=
\int_{i\af_+^*}\langle \cF \Phi(f_{t,X})(\lambda),\cF\Phi(g_{t,X})(\lambda)\rangle \,d\mu(\lambda)
\qquad (X\in\fa^{--}, t\geq 0)
$$
for all $f,g\in C_{c}(P\cdot z_{\empt})$.

\begin{lemma}\label{Lemma Z_empt Plancherel comparison}
Let $f\in C_c^{\infty}(P\cdot z_\empt)$ and let $S$ be a compact subset of  $\fa^{--}$. Then there exist constants $C>0$ and $\epsilon>0$ so that
\begin{equation}\label{eq difference of norms squared}
\bigg|
    \|f\|_{L^2(Z_\empt)}^2
    -\int_{i\af_+^*} \Big\|\sum_{w\in W} \cfunc(w\lambda)e^{-t w\lambda(X)} \, \pi_{\lambda}(f)\xi_\lambda^{w}\Big\|^{2}\,d\mu(\lambda)\bigg|
\leq C e^{-\e t}
\end{equation}
for all $X\in S$ and $t\geq0$.
\end{lemma}

The lemma is an explicated version of \cite[Corollary 8.3]{DKKS}. The proof for this result is rather involved; in the remainder of this section we will give some heuristics for the assertion.

Let $f,g\in C_{c}^{\infty}(Z_{\empt})$ and let $X\in \fa^{--}$.
For $t\geq 0$ we set $a_{t}=\exp(tX)$.  Then
\begin{align*}
&\Big\langle \cF \Phi(f_{t,X})(\lambda),\cF\Phi(g_{t,X})(\lambda)\Big\rangle
=\Big\langle \pi_{\lambda}\big(\Phi(f_{t,X})\big)\eta_{\lambda},\pi_{\lambda}\big(\Phi(g_{t,X})\big)\eta_{\lambda}\Big\rangle\\
&=e^{t\rho(X)}\int_{A}\int_{N}\int_{A}\int_{N} f(a_{1}n_{1}a_{t}^{-1})\overline{g_{t,X}}(a_{2}n_{2})
    \Big\langle \pi_{\lambda}(a_{1}n_{1})\eta_{\lambda},\pi_{\lambda}(a_{2}n_{2})\eta_{\lambda}\Big\rangle\,dn_{2}\,da_{2}\,dn_{1}\,da_{1}\\
&=e^{-t\rho(X)}\int_{A}\int_{N}\int_{A}\int_{N} f(a_{1}n_{1})\overline{g_{t,X}}(a_{2}n_{2})
    \Big\langle \pi_{\lambda}(a_{1}n_{1}a_{t})\eta_{\lambda},\pi_{\lambda}(a_{2}n_{2})\eta_{\lambda}\Big\rangle\,dn_{2}\,da_{2}\,dn_{1}\,da_{1}\\
&=e^{-t\rho(X)}\int_{A}\int_{N}\int_{A}\int_{N} f(a_{1}n_{1})\overline{g_{t,X}}(a_{1}n_{1}a_{2}n_{2})
    \Big\langle\pi_{\lambda}(a_{t})\eta_{\lambda},\pi_{\lambda}(a_{2}n_{2})\eta_{\lambda}\Big\rangle\,dn_{2}\,da_{2}\,dn_{1}\,da_{1}\\
&=e^{-t\rho(X)}\int_{A}\int_{N}\int_{A}\int_{N} f(a_{1}n_{1})\overline{g_{t,X}}(a_{1}n_{1}a_{2}n_{2})
    \Big(\pi_{\lambda}(a_{t})\eta_{\lambda},\pi_{-\lambda}(a_{2}n_{2})\eta_{-\lambda}\Big)\,dn_{2}\,da_{2}\,dn_{1}\,da_{1}\\
&=e^{-t\rho(X)}m_{v_{-\lambda,t},\eta_{\lambda}}(a_{t}),
\end{align*}
where
$$
v_{-\lambda,t}
=\int_{A}\int_{N}\int_{A}\int_{N} f(a_{1}n_{1})\overline{g_{t,X}}(a_{1}n_{1}a_{2}n_{2})
   \pi_{-\lambda}(a_{2}n_{2})\eta_{-\lambda}\,dn_{2}\,da_{2}\,dn_{1}\,da_{1}
\in\cH_{-\lambda}^{\infty}.
$$
Now using the constant term approximation from Theorem \ref{thm const} one can prove an estimate
$$
e^{-t\rho(X)}\Big|m_{v_{-\lambda,t},\eta_{\lambda}}(a_{t})- m_{v_{-\lambda,t},\eta_{\lambda}^\empt}(a_{t})\Big|
\leq Ce^{-\epsilon t}
$$
for a constant $C$ independent of $t$.
Now
\begin{align*}
&m_{v_{-\lambda,t},\eta_{\lambda}^\empt}(a_{t})\\
&=e^{t\rho(X)}\int_{A}\int_{N}\int_{A}\int_{N} f(a_{1}n_{1})\overline{g}(a_{1}n_{1}a_{2}n_{2}a_{t}^{-1})
    \Big(\pi_{\lambda}(a_{t})\eta_{\lambda}^{\empt},\pi_{-\lambda}(a_{2}n_{2})\eta_{-\lambda}\Big)\,dn_{2}\,da_{2}\,dn_{1}\,da_{1}\\
&=e^{-t\rho(X)}\int_{A}\int_{N}\int_{A}\int_{N} f(a_{1}n_{1})\overline{g}(a_{1}n_{1}a_{2}n_{2})
    \Big(\pi_{\lambda}(a_{t})\eta_{\lambda}^{\empt},\pi_{-\lambda}(a_{2}n_{2}a_{t})\eta_{-\lambda}\Big)\,dn_{2}\,da_{2}\,dn_{1}\,da_{1}\\
&=e^{-t\rho(X)}\int_{A}\int_{N}\int_{A}\int_{N} f(a_{1}n_{1})\overline{g}(a_{1}n_{1}a_{2}n_{2})
    \Big(\pi_{\lambda}(n_{2}^{-1}a_{2}^{-1}a_{t})\eta_{\lambda}^{\empt},\pi_{-\lambda}(a_{t})\eta_{-\lambda}\Big)\,dn_{2}\,da_{2}\,dn_{1}\,da_{1}\\
&=m_{w_{\lambda,t},\eta_{-\lambda}}(a_{t}),
\end{align*}
where
$$
w_{\lambda,t}
=e^{-t\rho(X)}\int_{A}\int_{N}\int_{A}\int_{N} f(a_{1}n_{1})\overline{g}(a_{1}n_{1}a_{2}n_{2})
   \pi_{\lambda}(n_{2}^{-1}a_{2}^{-1}a_{t})\eta_{\lambda}^{\empt}\,dn_{2}\,da_{2}\,dn_{1}\,da_{1}\in\cH_{\lambda}^{\infty}.
$$
Again the constant term approximation of Theorem \ref{thm const} may be used to prove an estimate
$$
e^{-t\rho(X)}\Big|m_{w_{\lambda,t},\eta_{-\lambda}}(a_{t})- m_{w_{\lambda,t},\eta_{-\lambda}^\empt}(a_{t})\Big|\leq Ce^{-\epsilon t}
$$
for some constant $C$ independent of $t$.
Note that
\begin{align*}
m_{w_{\lambda,t},\eta_{-\lambda}^\empt}(a_{t})
&=e^{-t\rho(X)}\Big(
\int_{A}\int_{N}\overline{g}(a_{2}n_{2})\pi_{\lambda}(a_{2}n_{2}a_{t})\eta_{\lambda}^{\empt},
\int_{A}\int_{N}f(a_{1}n_{1})\pi_{-\lambda}(a_{1}n_{1}a_{t})\eta_{-\lambda}^{\empt}
\Big)\\
&=e^{-t\rho(X)}\Big\langle \pi_{\lambda}(f)\pi_{\lambda}(a_{t})\eta_{\lambda}^{\empt},\pi_{\lambda}(g)\pi_{\lambda}(a_{t})\eta_{\lambda}^{\empt}\Big\rangle.
\end{align*}
In view of Corollary \ref{Cor Formula constant term} we have
\begin{align*}
e^{-t\rho(X)} \pi_{\lambda}(f)\pi_{\lambda}(a_{t})\eta_{\lambda}^{\empt}
=\sum_{w\in W} \cfunc(w\lambda)e^{-t w\lambda(X)} \, \pi_{\lambda}(f)\xi_\lambda^{w}.
\end{align*}
By taking $f=g$ we thus get the desired estimate for the difference
$$
\|\cF \Phi(f_{t,X})(\lambda)\|^{2}-\|\sum_{w\in W} \cfunc(w\lambda)e^{-t w\lambda(X)} \, \pi_{\lambda}(f)\xi_\lambda^{w}\|^{2}.
$$

\subsection{Averaging}
Recall that $\mu$ denotes the Plancherel measure for $L^{2}(Z)$ in (\ref{eq abstract Plancherel thm}).
The aim of this section is to prove the following proposition.

\begin{prop}\label{Prop Comparison of Planch decomps}
Let $f\in C_{c}^{\infty}(P\cdot z_{\empt})$. Then
$$
\|f\|_{L^2(Z_\empt)}^2
= \sum_{w\in W}\int_{i\af_+^*} |\cfunc(w\lambda)|^2 \|\cF_{\empt}f(\lambda)\xi_\lambda^{w}\|^2  \,d\mu(\lambda).
$$
\end{prop}

We will prove the proposition by averaging (\ref{eq difference of norms squared}) over $t$ and $X$ and taking a limit.
We first prove a lemma.

\begin{lemma}\label{Lemma Oscillatory means}
Let $\cV$ be a finite dimensional representation of $\fa$ so that every $X\in \fa$ acts with only purely imaginary eigenvalues. Let $X_{1},\dots, X_{m}$ be a basis of $\fa$. Let $c$ be any positive irrational real number and set $X_{j+m}=cX_{j}$ for $1\leq j\leq m$. Let further $\cH$ be an inner product space and $F:\cV\to \cH$ a linear map.  For $n\in\N$ we define the set
$$
S_{n}
:=\Big\{\sum_{j=1}^{2m}t_{j}X_{j}\in \fa\mid t_{j}\in\N \text{ and } n+1\leq t_{j}\leq 2n \text{ for all }1\leq j\leq 2m\Big\}
$$
and the map
$$
A_{n}:\cV\to [0,\infty], \quad
    v\mapsto \frac{1}{n^{2m}} \sum_{X\in S_{n}}\|F (e^{X}\cdot v)\|^{2}.
$$
Then the following hold.
\begin{enumerate}[(i)]
\item\label{Lemma Oscillatory means - item 1} If  $v=\sum_{k=1}^{N}v_{k}$ is a decomposition of $v\in \cV$ into a sum of joint eigenvectors $v_{k}$ with distinct eigencharacters, then
$$
\lim_{n\to\infty}A_{n}(v)
=\sum_{k=1}^{N}\|F v_{k}\|^{2}.
$$
\item\label{Lemma Oscillatory means - item 2} Assume that $F$ is injective. If $v\in \cV\setminus\{0\}$ does not decompose into a sum of joint eigenvectors, then
$$
\lim_{n\to\infty}A_{n}(v)=\infty.
$$
\end{enumerate}
\end{lemma}

\begin{proof}
We begin with (\ref{Lemma Oscillatory means - item 1}).
Let $v\in \cV$.
Assume that there exist $v_{1},\dots, v_{N}\in \cV$ and distinct $\lambda_{1},\dots, \lambda_{N}\in i\fa^{*}$ so that
$$
e^{X}\cdot v
= \sum_{k=1}^{N}e^{\lambda_{k}(X)} v_{k}\qquad (X\in \fa).
$$
We have
\begin{align*}
A_{n}(v)
&= \frac{1}{n^{2m}} \sum_{t_{1}=n+1}^{2n}\sum_{t_{2}=n+1}^{2n}\dots \sum_{t_{2m}=n+1}^{2n}\sum_{k=1}^{N}\sum_{l=1}^{N}
    e^{\sum_{j=1}^{2m}t_{j}\big(\lambda_{k}(X_{j})-\lambda_{l}(X_{j})\big)}\langle Fv_{k},Fv_{l}\rangle\\
&=\sum_{k=1}^{N}\|Fv_{k}\|^{2}
    +2\Re\bigg(\sum_{1\leq k<l\leq N}\Big(\prod_{j=1}^{2m}\frac{1}{n}\sum_{t=n+1}^{2n}e^{t\big(\lambda_{k}(X_{j})-\lambda_{l}(X_{j})\big)}\Big)
    \langle Fv_{k},Fv_{l}\rangle\bigg).
\end{align*}
Note that for every $j$, $k$, $l$ and $n$
$$
\Big|\frac{1}{n}\sum_{t=n+1}^{2n}e^{t\big(\lambda_{k}(X_{j})-\lambda_{l}(X_{j})\big)}\Big|
\leq \frac{1}{n}\sum_{t=n+1}^{2n}\Big|e^{t\big(\lambda_{k}(X_{j})-\lambda_{l}(X_{j})\big)}\Big|
=1.
$$
The elements $X_{j}$ have been chosen so that for every pair $k$ and $l$ with $k\neq l$ there exists a $j$ so that $e^{\lambda_{k}(X_{j})-\lambda_{l}(X_{j})}\neq 1$. For this $j$ we have
$$
\frac{1}{n} \sum_{t=n+1}^{2n}e^{t\big(\lambda_{k}(X_{j})-\lambda_{k}(X_{j})\big)}
=\frac{1}{n}\frac{e^{(n+1)\big(\lambda_{k}(X_{j})-\lambda_{k}(X_{j})\big)}-e^{(2n+1)\big(\lambda_{k}(X_{j})-\lambda_{k}(X_{j})\big)}}
    {1-e^{\big(\lambda_{k}(X_{j})-\lambda_{k}(X_{j})\big)}}
\to 0
$$
for $n\to\infty$. This proves (\ref{Lemma Oscillatory means - item 1}).

We move on to (\ref{Lemma Oscillatory means - item 2}). Let $\lambda_{1},\dots, \lambda_{N}\in i\fa^{*}$ be the distinct generalized joint eigencharacters occurring in $v$.
There exists for every $1\leq k\leq N$ and every multi-index $\mu$ in $m$-variables  a generalized eigenvector $v_{k,\mu}$ with eigenvalue $\lambda_{k}$, so that
$$
e^{X}\cdot v
=\sum_{k=1}^{N}e^{\lambda_{k}(X)} \sum_{\mu}t^{\mu}v_{k,\mu}
\qquad\big(X=\sum_{j=1}^{m}t_{j}X_{j}\in \fa\big).
$$
We may and will assume that the non-zero $v_{k, \mu}$ are linearly independent. Moreover, by the assumption on $v$ we have $v_{k,\mu}\neq 0$ for at least one pair $(k,\mu)$ with $|\mu|\geq1$. Let $\Xi=\{(k,\mu)\mid v_{k,\mu}\neq 0\}$.
Therefore, the map
$$
\R^{\Xi}\to[0,\infty),\quad \gamma\mapsto \Big\|\sum_{(k,\mu)\in \Xi}\gamma_{k,\mu}Fv_{k,\mu}\Big\|
$$
is a norm on $\R^{\Xi}$. By equivalence of norms there exists a $C>0$ so that
$$
\Big\|\sum_{(k,\mu)\in\Xi}\gamma_{k,\mu}Fv_{k,\mu}\Big\|\geq C\|\gamma\|\qquad(\gamma\in\R^{\Xi}).
$$
Now
\begin{align*}
A_{n}(v)
&= \frac{1}{n^{2m}} \sum_{t_{1}=n+1}^{2n}\sum_{t_{2}=n+1}^{2n}\dots \sum_{t_{2m}=n+1}^{2n}\Big\|\sum_{(k,\mu)\in\Xi}e^{\sum_{j=1}^{m}t_{j}\lambda_{k}(X_{j})}t^{\mu}Fv_{k,\mu}\Big\|^{2}\\
&\geq C^{2}\frac{1}{n^{2m}} \sum_{t_{1}=n+1}^{2n}\sum_{t_{2}=n+1}^{2n}\dots \sum_{t_{2m}=n+1}^{2n}
    \Big\|\Big(e^{\sum_{j=1}^{m}t_{j}\lambda_{k}(X_{j})}t^{\mu}\Big)_{(k,\mu)\in\Xi}\Big\|^{2}.
\end{align*}
The latter becomes arbitrarily large for $n\to\infty$ as there exists a pair $(k, \mu)\in\Xi$ with $|\mu|\geq1$.
\end{proof}

\begin{proof}[Proof of Proposition \ref{Prop Comparison of Planch decomps}]
Let $X_{1},\dots, X_{m}\in \fa^{--}$ be a basis of $\fa$. Let $c$ be any positive irrational real number and set $X_{j+m}=cX_{j}$ for $1\leq j\leq m$.
We define 
$$
S
:=\Big\{\sum_{j=1}^{2m}t_{j}X_{j}\in \fa\mid 1\leq t_{j}\leq 2 \text{ for all }1\leq j\leq 2m\Big\}
$$
and set for  $n\in\N$
$$
S_{n}
:=\Big\{\sum_{j=1}^{2m}t_{j}X_{j}\in \fa\mid t_{j}\in\N\text{ and }n+1\leq t_{j}\leq 2n \text{ for all }1\leq j\leq 2m\Big\}.
$$
Note that $S$ is a compact subset of $\fa^{--}$ and $S_{n}\subseteq n S$.
By Lemma \ref{Lemma Z_empt Plancherel comparison} there exist constants $C>0$ and $\epsilon>0$ (independent of $n$) so that
\begin{align}
\nonumber&\bigg|
    \|f\|_{L^2(Z_\empt)}^2
    - \int_{i\af_+^*}
     \frac{1}{n^{2m}} \sum_{X\in S_{n}}\Big\|\sum_{w\in W} \cfunc(w\lambda)e^{-w\lambda(X)} \,
      \pi_{\lambda}(f)\xi_\lambda^{w}\Big\|^{2}\,d\mu(\lambda)\bigg|\\
\nonumber&\leq\frac{1}{n^{2m}} \sum_{X\in S_{n}}\bigg|
    \|f\|_{L^2(Z_\empt)}^2
    -\int_{i\af_+^*} \Big\|\sum_{w\in W} \cfunc(w\lambda)e^{-w\lambda(X)} \, \pi_{\lambda}(f)\xi_\lambda^{w}\Big\|^{2}\,d\mu(\lambda)\bigg|\\
\nonumber&=\frac{1}{n^{2m}} \sum_{X\in \frac{1}{n}S_{n}}\bigg|
    \|f\|_{L^2(Z_\empt)}^2
    -\int_{i\af_+^*} \Big\|\sum_{w\in W} \cfunc(w\lambda)e^{-nw\lambda(X)} \, \pi_{\lambda}(f)\xi_\lambda^{w}\Big\|^{2}\,d\mu(\lambda)\bigg|\\
\label{eq Limit}
&\leq C e^{-\epsilon n}
\to 0\qquad (n\to\infty).
\end{align}

For $\lambda\in i\fa^{*}_{+}$, we define
$$
\cQ_{\lambda}(f)
:=\liminf_{n\to\infty}\frac{1}{n^{2m}}\sum_{X\in S_{n}}\Big\|\sum_{w\in W} \cfunc(w\lambda)e^{-w\lambda(X)} \,
      \pi_{\lambda}(f)\xi_\lambda^{w}\Big\|^{2}\in[0,\infty].
$$
By Fatou's lemma and (\ref{eq Limit})
\begin{align*}
&\int_{i\fa^{*}_{+}}\cQ_{\lambda}(f)\,d\mu(\lambda)\\
&\qquad\leq \liminf_{n\to\infty}\int_{i\fa^{*}_{+}}\frac{1}{n^{2m}}\sum_{X\in S_{n}}\Big\|\sum_{w\in W} \cfunc(w\lambda)e^{- w\lambda(X)} \,
      \pi_{\lambda}(f)\xi_\lambda^{w}\Big\|^{2}\,d\mu(\lambda)
=\|f\|_{L^{2}(Z_{\empt})}^{2}.
\end{align*}
We define the set
$$
\fa^{*}_{\good}
:=\Big\{\lambda\in\fa^{*}_{+}\mid \cQ_{i\lambda}(f)=\sum_{w\in W}|\cfunc(iw\lambda)|^{2}\, \|\pi_{i\lambda}(f)\xi_{i\lambda}^{w}\|^{2} \text{ for all }f\in C_{c}^{\infty}(P\cdot z_{\empt})\Big\}.
$$

Since $C_{c}^{\infty}(P\cdot z_{\empt})$ is dense in $L^{2}(Z_{\empt})$, it follows from the Plancherel decomposition Theorem \ref{Thm Planch thm for Z_empt} for $Z_{\empt}$ that for every $\lambda_{0}\in \fa^{*}_{+}$ there exists an $f\in C_{c}^{\infty}(P\cdot z_{\empt})$ and an open neighborhood $U\subseteq\fa^{*}_{+}$ of $\lambda_{0}$ so that the restriction of $\pi_{i\lambda}(f)$ to $(\cH_{i\lambda}^{-\infty})^{M\overline{N}}$ is injective for all $\lambda\in U$.
Now
\begin{align*}
&\sum_{w\in W}\int_{iU\cap i\fa^{*}_{\good}}|\cfunc(w\lambda)|^{2}\, \|\pi_{\lambda}(f)\xi_\lambda^{w}\|^{2}\,d\mu(\lambda)
+\int_{iU\setminus i\fa^{*}_{\good}}\cQ_{\lambda}(f)\,d\mu(\lambda)\\
&\qquad=\int_{iU}\cQ_{\lambda}(f)\,d\mu(\lambda)
\leq\|f\|_{L^{2}(Z_{\empt})}^{2}.
\end{align*}
By Lemma \ref{Lemma Oscillatory means} we have $\cQ_{\lambda}(f)=\infty$ for all $\lambda\in iU\setminus i\fa^{*}_{\good}$. It therefore follows that $\mu(iU\setminus i\fa^{*}_{\good})=0$. Since any $\lambda_{0}\in \fa^{*}_{+}$ has an open neighborhood $U$ such that $\mu(iU\setminus i\fa^{*}_{\good})=0$, it follows that $\mu(i\fa_{+}\setminus i\fa^{*}_{\good})=0$.
We thus obtain
$$
\sum_{w\in W}\int_{i\fa^{*}_{+}}|\cfunc(w\lambda)|^{2}\, \|\pi_{\lambda}(f)\xi_\lambda^{w}\|^{2}\,d\mu(\lambda)
\leq\|f\|_{L^{2}(Z_{\empt})}^{2}
$$
for all $f\in C_{c}^{\infty}(P\cdot z_{\empt})$.

In view of the triangle inequality we have
$$
\frac{1}{n^{2m}} \sum_{X\in S_{n}}\Big\|\sum_{w\in W} \cfunc(w\lambda)e^{-w\lambda(X)} \,
      \pi_{\lambda}(f)\xi_\lambda^{w}\Big\|^{2}
\leq \sum_{w\in W} |\cfunc(w\lambda)|^{2}\, \|\pi_{\lambda}(f)\xi_\lambda^{w}\|^{2}
$$
for every $f\in C_{c}^{\infty}(P\cdot z_{\empt})$, $\lambda\in i\fa^{*}_{+}$ and $n\in \N$. It thus follows from Lebesgue's dominated convergence theorem and (\ref{eq Limit}) that
\begin{align*}
&\sum_{w\in W}\int_{i\fa^{*}_{+}}|\cfunc(w\lambda)|^{2}\, \|\pi_{\lambda}(f)\xi_\lambda^{w}\|^{2}\,d\mu(\lambda)\\
&\qquad=\lim_{n\to\infty} \int_{i\fa^{*}_{+}}\frac{1}{n^{2m}}\sum_{X\in S_{n}}\Big\|\sum_{w\in W} \cfunc(w\lambda)e^{- w\lambda(X)} \,
      \pi_{\lambda}(f)\xi_\lambda^{w}\Big\|^{2}\,d\mu(\lambda)
=\|f\|_{L^{2}(Z_{\empt})}^{2}.
\end{align*}
\end{proof}

\subsection{Proof of Theorem \ref{thm planch}}
For $w\in W$ we define the measure 
$$
d\mu_{w}(\lambda):=|\cfunc(w\lambda)|^{2}d\mu(\lambda).
$$
We decompose the multiplicity space $\cM_{\lambda}$ as
$$
\cM_{\lambda}=\bigoplus_{w\in W}\cM_{\lambda}^{w},
$$
where $\cM_{\lambda}^{w}:=\big(\C\xi_{\lambda}^{w}\big)'$.
The space $C_{c}^{\infty}(P\cdot z_{\empt})$ is a dense subspace of $L^{2}(Z_{\empt})$. It follows from  Proposition \ref{Prop Comparison of Planch decomps} that $\cF_{\empt}$ extends to a $G$-equivariant unitary map
$$
\big(L, L^2(Z_\empt)\big)\to
    \left(\bigoplus_{w\in W}\int_{i\af_+^*}^\oplus \pi_\lambda\otimes {\rm id}_{\cM_{\lambda}^{w}}\,d\mu_{w}(\lambda), \bigoplus_{w\in W} \int_{i\af^*_{+}}^\oplus \cH_\lambda \otimes \cM_{\lambda}^{w}\,d\mu_{w}(\lambda)\right).
$$
The Plancherel measure is unique.  See for example \cite[Theorem C.I \& Theorem II.6]{Penney}. Therefore the Plancherel decomposition for $\big(L,L^{2}(Z_{\empt})\big)$  in Theorem \ref{Thm Planch thm for Z_empt} implies that 
\begin{equation}\label{eq Equality of measures}
d\mu_{w}(\lambda)=d\lambda \qquad (w\in W).
\end{equation}
In particular, since the $c$-functions are continuous, it follows that $|\cfunc(w\lambda)|^{2}=|\cfunc(\lambda)|^{2}$ for all $w\in W$ and $\lambda\in i\fa^{*}$. These are the Maass-Selberg relations. Moreover, (\ref{eq Equality of measures}) implies that 
$$
d\mu(\lambda)=\frac{d\lambda}{|\cfunc(\lambda)|^{2}}.
$$
The abstract Plancherel decomposition (\ref{eq abstract Plancherel thm}) for $L^{2}(Z)$ now explicates to the assertion that the Fourier transform $\cF$
is a $G$-equivariant unitary isomorphism
$$
\big(L, L^2(Z)\big)\to \left(\int_{i\af^*_{+}}^\oplus \pi_\lambda \,\frac{ d\lambda}{|\cfunc(\lambda)|^{2}},
     \int_{i\af^*_{+}}^\oplus \cH_\lambda\,\frac{ d\lambda}{|\cfunc(\lambda)|^{2}}\right).
 $$
This proves Harish-Chandra's Theorem \ref{thm planch}.

\end{document}